\documentclass[10pt]{article}

\usepackage[utf8]{inputenc}
\usepackage[a4paper, tmargin=1.0in,bmargin=1.0in, lmargin=1.in, rmargin=1.in]{geometry}
\usepackage{hyperref}
\usepackage{amsmath}
\usepackage{amssymb}
\usepackage{amsthm}

\usepackage[T1]{fontenc}
\usepackage[english]{babel}
\usepackage{xcolor}

\usepackage{stackrel}
\usepackage{bm}
\usepackage{stmaryrd}
\usepackage{accents} 
\usepackage{mathtools} 
\usepackage{cases} 
\usepackage{cleveref}

\usepackage{tikz}
\usepackage{pgfplots}
\usetikzlibrary{calc}

\definecolor{myViolet}{RGB}{97, 30, 83}
\definecolor{myRed}{RGB}{203, 26, 79}
\definecolor{myOrange}{RGB}{245, 136, 95}

\newtheorem{remark}{Remark}
\newtheorem{example}{Example}

\newcommand{\abstractProblemFunction}{g}
\newcommand{\timeVariable}{t}
\newcommand{\basicBrickSolver}{\vectorial{\psi}}
\newcommand{\timeStep}{\Delta \timeVariable}

\newcommand{\perturbationVectorField}{d}
\newcommand{\bigO}[1]{\mathcal{O}(#1)}
\newcommand{\globalScheme}{\vectorial{\phi}}
\newcommand{\coefficientOuterStep}{\alpha}
\newcommand{\coefficientInnerStep}{\beta}
\newcommand{\relatives}{\mathbb{Z}}

\newcommand{\indicesSpace}{j}

\newcommand{\numberOuterSteps}{n}
\newcommand{\matricial}[1]{\bm{#1}}
\newcommand{\discrete}[1]{\mathsf{#1}}
\newcommand{\vectorial}[1]{\bm{#1}}
\newcommand{\transport}{\matricial{\discrete{T}}}
\newcommand{\relaxation}[1]{\matricial{\discrete{R}}_{#1}}

\newcommand{\reals}{\mathbb{R}}
\newcommand{\spaceVariable}{x}
\newcommand{\relaxationParameter}{\omega}

\newcommand{\conservedMomentsSystemConservationLaws}{u}
\newcommand{\fluxSystemConservationLaws}{\varphi}
\newcommand{\numberConservationLaws}{M}
\newcommand{\nonConservedMomentsSystemConservationLaws}{v}
\newcommand{\energyNonConservedMomentsSystemConservationLaws}{w}
\newcommand{\relaxationTime}{\epsilon}
\newcommand{\kineticVelocity}{V}

\newcommand{\splittingTimeStep}{\Delta \timeVariable}

\newcommand{\spatialDimensionality}{d}
\newcommand{\distributionFunctionLetter}{f}
\newcommand{\atEquilibrium}{\textnormal{eq}}
\newcommand{\spaceStep}{\Delta \spaceVariable}

\newcommand{\indiceSpace}{k}

\newcommand{\identityOperator}{\matricial{\discrete{Id}}}
\newcommand{\idEst}{\emph{i.e.}}

\newcommand{\finalTime}{T}
\newcommand{\lbmSchemeVectorial}[3]{$\text{D}_{#1}\text{Q}_{#2}^{#3}$}
\newcommand{\canonicalBasisVector}{\vectorial{e}}
\newcommand{\transpose}[1]{#1^T}
\newcommand{\spaceLattice}[1]{\spaceStep\relatives^{#1}}
\newcommand{\numberShiftsTransport}{\kappa}

\newcommand{\naturalsWithoutZero}{\mathbb{N}^{*}}
\newcommand{\rationals}{\mathbb{Q}}
\newcommand{\transportVelocity}{a}
\newcommand{\timeShiftOperator}{z}
\newcommand{\determinant}{\text{det}}
\newcommand{\fourierTransformed}[1]{\hat{#1}}
\newcommand{\trace}{\text{tr}}
\newcommand{\frequency}{\xi}

\newcommand{\entropy}{S}
\newcommand{\entropyFlux}{{G}}

\newcommand{\indicesVelocity}{k}
\newcommand{\numberVelocities}{q}
\newcommand{\integerInterval}[2]{\llbracket #1, #2 \rrbracket}
\newcommand{\microscopicEntropyLetter}{\Sigma}
\newcommand{\heightShallowWater}{h}
\newcommand{\velocityShallowWater}{u}
\newcommand{\gravityShallowWater}{g}
\newcommand{\confer}{\emph{cf.}}
\newcommand{\densityEuler}{\rho}
\newcommand{\velocityXEuler}{u}
\newcommand{\velocityYEuler}{v}
\newcommand{\energyEuler}{E}
\newcommand{\pressureEuler}{p}
\newcommand{\gasConstant}{\gamma}

\usepackage[normalem]{ulem}

\newcommand{\greatersim}{\raisebox{-0.13cm}{~\shortstack{$>$ \\[-0.07cm]$\sim$}}~}
\ifpdf
\hypersetup{
  pdftitle={Fourth-order entropy-stable lattice Boltzmann schemes for hyperbolic systems},
  pdfauthor={T. Bellotti, P. Helluy, L. Navoret}
}
\fi

\newtheorem{proposition}{Proposition}%

\author{T. Bellotti, P. Helluy, L. Navoret}

\newcommand{\TheTitle}{Fourth-order entropy-stable lattice Boltzmann schemes  for hyperbolic systems}

\title{{\TheTitle}}

\author{
  Thomas Bellotti\thanks{IRMA, Universit\'e de Strasbourg, 67000 Strasbourg, France.}
  \and
  Philippe Helluy\footnotemark[1]
  \and
  Laurent Navoret\footnotemark[1]
}

\begin{document}

\maketitle

\begin{abstract}
    We present a novel framework for the development of fourth-order lattice Boltzmann schemes to tackle multidimensional nonlinear systems of conservation laws. 
    As for other numerical schemes for hyperbolic problems, high-order accuracy applies only to smooth solutions.
    Our numerical schemes preserve two fundamental characteristics inherent in classical lattice Boltzmann methods: a local relaxation phase and a transport phase composed of elementary shifts on a Cartesian grid.
    Achieving fourth-order accuracy is accomplished through the composition of second-order time-symmetric basic schemes utilizing rational weights. This enables the representation of the transport phase in terms of elementary shifts. Introducing local variations in the relaxation parameter during each stage of relaxation ensures entropy stability of the schemes. This not only enhances stability in the long-time limit but also maintains fourth-order accuracy.
    To validate our approach, we conduct comprehensive testing on scalar equations and systems in both one and two spatial dimensions. 
\end{abstract}

    
\begin{amsCat}
    76M28, 65M99, 65M12
\end{amsCat}


\section*{Introduction}

Lattice Boltzmann schemes \cite{higuera1989boltzmann} have gained acclaim for their computational efficiency and ease of use on modern computer architectures (\emph{e.g.} GPUs), owing to their distinctive structure, comprising a local collision/relaxation phase and a linear transport phase. The latter is constructed through shifts of data on a regular Cartesian grid. Despite their recent application in simulating non-linear systems of conservation laws \cite{dubois2014simulation, graille2014approximation, drui2019analysis,bellotti2022multiresolution,  bellotti2022multidimensional, baty:hal-02965967, guillon:hal-03986533}, these methods exhibit lower accuracy for such problems compared to more conventional approaches like Finite Volume and Discontinuous Galerkin methods.

While the attainment of second-order accuracy in lattice Boltzmann schemes is well-understood, achieved by setting relaxation parameters to two \cite{dellar2013interpretation, dubois2022nonlinear, bellotti2023truncation}, obtaining third and fourth-order accuracy proves to be a more intricate challenge \cite{dubois2022nonlinear}. The ability to increase the order is not guaranteed \emph{a priori}---especially for non-linear equations---and depends on the specific lattice Boltzmann scheme in use. When possible, higher accuracy is attained by delicately tuning equilibria that do not contribute to consistency at the leading order, a process that can be complex. Moreover, it is challenging to ascertain the stability of the scheme under such modifications. Consequently, the only fourth-order schemes identified so far only address linear scalar equations in 1D \cite{chen2023fourth, bellotti2023influence}, with minimal practical significance, the linear diffusion equation in 2D \cite{chen2024general}, and a specific kind of coupled Burgers' equations \cite{chen2023cole}.
Finally, a third-order time-accurate / sixth-order space-accurate scheme \cite{lin2021multiple} exists to tackle 1D linear diffusion equations.

This contribution aims at establishing a comprehensive framework for constructing fourth-order kinetic schemes for non-linear systems of multi-dimensional conservation laws. To achieve this objective, a departure from standard lattice Boltzmann schemes is necessary. Nevertheless, the numerical schemes to be developed maintain the two keys to success of any standard lattice Boltzmann scheme, namely the locality of the collision phase and a transport phase made up of simple shifts, thus retain their notorious efficiency, compared to standard kinetic schemes \cite{lafitte2017high, abgrall2023arbitrarily}.
The essential idea to increase the order of the schemes is to allow both forward and backward steps in time.

In the whole paper, we deal mostly with smooth solutions: the development of limiting procedures in the context of lattice Boltzmann schemes---a recently emerging topic \cite{kozhanova4755400hybrid}---remains an open and fundamental question worth thorough discussions that we do not address in the present contribution.
Matching the discussion concerning limiters for Finite Volume schemes, limiters can be divided into \emph{a priori} limiters, such as slope limiters, see \cite[Chapter 6]{leveque2002finite} for a general overview, and \emph{a posteriori} ones, \emph{e.g.} \cite{clain2011high}.
The latter correct the solution after a tentative first guess with pathologies has been computed.
Ignoring limiters, the present contribution aims at being a proof of concept of a new way to construct high-order lattice Boltzmann schemes for general hyperbolic problems.

The paper is organized as follows. In \Cref{sec:targetAndRelaxation}, we introduce the system of conservation laws addressed in this paper, along with its relaxation approximation, facilitating the handling of non-linearity. \Cref{sec:kineticSchemes} outlines our numerical strategy, beginning with a conventional lattice Boltzmann scheme, followed by a time-symmetrization step \cite{coulette2019high}, and eventual composition to achieve fourth-order accuracy \cite{mclachlan2002splitting}. We finish on brief considerations about stability, specifically in terms of the $L^2$ norm in a basic scalar linear setting.
A first batch of numerical experiments, presented in \Cref{sec:experimentsOrder}, empirically confirms the theoretical predictions.
In \Cref{sec:entropyStability}, we introduce a method for adjusting the relaxation parameter to ensure entropy stability for the numerical scheme.
The need for this procedure and the fact that it does not alter the order of the scheme are studied by means of new numerical experiments in \Cref{sec:experimentsEntropy}.
Finally, in \Cref{sec:Variations}, we propose, study, and test several variations on our fourth-order scheme.
We eventually conclude in \Cref{sec:conclusions}.

\section{Target system of conservation laws and relaxation approximation}\label{sec:targetAndRelaxation}

\subsection{Target system of conservation laws}

We aim at approximating the solution of the system on $\vectorial{\conservedMomentsSystemConservationLaws} : \reals \times \reals^{\spatialDimensionality} \to \reals^{\numberConservationLaws}$:
\begin{equation}\label{eq:systemConservationLaws}
    \partial_{\timeVariable} \vectorial{\conservedMomentsSystemConservationLaws} + \sum_{\indicesSpace = 1}^{\spatialDimensionality} \frac{\partial}{\partial \spaceVariable_{\indicesSpace}}  \vectorial{\fluxSystemConservationLaws}^{\indicesSpace}(\vectorial{\conservedMomentsSystemConservationLaws}) = 0,
\end{equation}
where $\vectorial{\fluxSystemConservationLaws}^{\indicesSpace}: \reals^{\numberConservationLaws} \to \reals^{\numberConservationLaws}$ for $\indicesSpace \in \integerInterval{1}{\spatialDimensionality}$ are smooth and possibly non-linear fluxes, see \cite{godlewski2013numerical}.
To give a simple example, taking $\spatialDimensionality = 1$, $\numberConservationLaws = 1$,  and $\fluxSystemConservationLaws(\conservedMomentsSystemConservationLaws) = \tfrac{1}{2}\conservedMomentsSystemConservationLaws^2$ yields the inviscid Burgers' equation.
We assume that \eqref{eq:systemConservationLaws} admits a Lax entropy--entropy fluxes pair $(\entropy, {\entropyFlux}^1, \dots, {\entropyFlux}^{\spatialDimensionality})$, with $\entropy : \reals^{\numberConservationLaws}\to \reals$ and ${\entropyFlux}^{\indicesSpace} : \reals^{\numberConservationLaws} \to \reals$ for $\indicesSpace \in \integerInterval{1}{\spatialDimensionality}$, such that ${\nabla_{\vectorial{\conservedMomentsSystemConservationLaws}} \vectorial{\fluxSystemConservationLaws}}^{\indicesSpace} \nabla_{\vectorial{\conservedMomentsSystemConservationLaws}} \entropy  = \nabla_{\vectorial{\conservedMomentsSystemConservationLaws}} \entropyFlux^{\indicesSpace}$ with $\entropy$ convex.
Further properties on this construction can be found in \cite{bouchut1999construction, bouchut2003entropy}.

\subsection{Relaxation systems}

In order to isolate the non-linearity of the fluxes appearing in \eqref{eq:systemConservationLaws} into a local relaxation term which is easily tractable, we consider the following discrete-velocity BGK relaxation systems \cite{bouchut1999construction, aregba2000discrete, bouchut2003entropy} on the distribution functions $\vectorial{\distributionFunctionLetter}_{1}, \dots, \vectorial{\distributionFunctionLetter}_{\numberVelocities} : \reals \times \reals^{\spatialDimensionality} \to \reals^{\numberConservationLaws}$, under the form 
\begin{equation}\label{eq:relaxationKineticSystem}
    \partial_{\timeVariable} \vectorial{\distributionFunctionLetter}_{\indicesVelocity} + \sum_{\indicesSpace = 1}^{\spatialDimensionality} \kineticVelocity_{\indicesVelocity}^{\indicesSpace} \frac{\partial \vectorial{\distributionFunctionLetter}_{\indicesVelocity}}{\partial\spaceVariable_{\indicesSpace}} = -\frac{1}{\relaxationTime} (\vectorial{\distributionFunctionLetter}_{\indicesVelocity} - \vectorial{\distributionFunctionLetter}^{\atEquilibrium}_{\indicesVelocity}(\vectorial{\conservedMomentsSystemConservationLaws} )), \qquad \indicesVelocity \in \integerInterval{1}{\numberVelocities}.
\end{equation}
Here, $\numberVelocities \geq 2$ is the number of discrete velocities, which are $\vectorial{\kineticVelocity}_{\indicesVelocity} \in \reals^{\spatialDimensionality}$, $\vectorial{\conservedMomentsSystemConservationLaws} = \sum_{\indicesVelocity = 1}^{\indicesVelocity = \numberVelocities} \vectorial{\distributionFunctionLetter}_{\indicesVelocity}$, and we indicate the relaxation time by $\relaxationTime > 0$.
Moreover, the equilibria $\vectorial{\distributionFunctionLetter}^{\atEquilibrium}_{\indicesVelocity}$ are non-linear functions of $\vectorial{\conservedMomentsSystemConservationLaws} $,and fulfill the compatibility relations 
\begin{equation}\label{eq:constraintsKinetic}
    \vectorial{\conservedMomentsSystemConservationLaws} = \sum_{\indicesVelocity = 1}^{\numberVelocities} \vectorial{\distributionFunctionLetter}^{\atEquilibrium}_{\indicesVelocity}(\vectorial{\conservedMomentsSystemConservationLaws}), \qquad \vectorial{\fluxSystemConservationLaws}^{\indicesSpace}(\vectorial{\conservedMomentsSystemConservationLaws}) = \sum_{\indicesVelocity = 1}^{\numberVelocities}{\kineticVelocity}_{\indicesVelocity}^{\indicesSpace} \vectorial{\distributionFunctionLetter}^{\atEquilibrium}_{\indicesVelocity}(\vectorial{\conservedMomentsSystemConservationLaws}), \quad \indicesSpace \in \integerInterval{1}{\spatialDimensionality},
\end{equation}
under which, the formal limit $\relaxationTime \to 0^+$ gives that $\vectorial{\distributionFunctionLetter}_{\indicesVelocity} \approx \vectorial{\distributionFunctionLetter}^{\atEquilibrium}_{\indicesVelocity}$ and that the sum of the distribution functions under \eqref{eq:relaxationKineticSystem} approximates \eqref{eq:systemConservationLaws}, see \cite{aregba2000discrete,lafitte2017high}.

\newcommand{\kineticEntropy}[1]{s_{#1}}
\newcommand{\legendreTransformed}[1]{#1^{*}}
\newcommand{\almostLegendreTransformedSymbol}{\star}
\newcommand{\almostLegendreTransformed}[1]{#1^{\almostLegendreTransformedSymbol}}
\newcommand{\kineticEntropyLegendre}[1]{\legendreTransformed{s}_{#1}}

\newcommand{\dualVariable}{p}

For future use, we introduce the microscopic entropy \cite{dubois2013stable}, given by the sum of the kinetic entropies: $\microscopicEntropyLetter(\vectorial{\distributionFunctionLetter}_{1}, \dots, \vectorial{\distributionFunctionLetter}_{\indicesVelocity}) =  \sum_{\indicesVelocity = 1}^{\indicesVelocity = \numberVelocities} \kineticEntropy{\indicesVelocity}(\vectorial{\distributionFunctionLetter}_{\indicesVelocity})$, where the kinetic entropies $\kineticEntropy{1}, \dots, \kineticEntropy{\numberVelocities} : \reals^{\numberConservationLaws} \to \reals$ are convex functions of their argument under the so-called characteristic condition.
We also have
\begin{equation}\label{eq:entropyDefinitionMinimization}
    \entropy(\vectorial{\conservedMomentsSystemConservationLaws}) = \min_{\vectorial{\conservedMomentsSystemConservationLaws} = \sum_{\indicesVelocity = 1}^{\indicesVelocity = \numberVelocities} \vectorial{\distributionFunctionLetter}_{\indicesVelocity}} \microscopicEntropyLetter(\vectorial{\distributionFunctionLetter}_{1}, \dots, \vectorial{\distributionFunctionLetter}_{\indicesVelocity}) = \microscopicEntropyLetter(\vectorial{\distributionFunctionLetter}_{1}^{\atEquilibrium}(\vectorial{\conservedMomentsSystemConservationLaws}), \dots, \vectorial{\distributionFunctionLetter}_{\indicesVelocity}^{\atEquilibrium}(\vectorial{\conservedMomentsSystemConservationLaws})),
\end{equation}
meaning that the entropy $\entropy$ stems from a constrained optimization of the microscopic entropy $\microscopicEntropyLetter$, and that the minimum is reached on the equilibrium.
Furthermore, equation \eqref{eq:entropyDefinitionMinimization} tells us that the entropy is an inf-convolution of the kinetic entropies \cite{guillon:hal-03986533}, which translates---thanks to the Legendre-Fenchel transform \cite{Rockafellar+1970} that we shall indicate by a $\legendreTransformed{}$---into $\legendreTransformed{\entropy} = \sum_{\indicesVelocity = 1}^{\indicesVelocity = \numberVelocities} \kineticEntropyLegendre{\indicesVelocity}$.
We also additionally request that \cite{dubois2013stable}
\begin{equation}\label{eq:transformationEntropyFlux}
    {{\entropyFlux}}^{\indicesSpace, \almostLegendreTransformedSymbol} := \vectorial{\fluxSystemConservationLaws}^{\indicesSpace}(\vectorial{\conservedMomentsSystemConservationLaws}(\vectorial{\dualVariable})) \cdot \vectorial{\dualVariable} - {\entropyFlux}^{\indicesSpace}(\vectorial{\conservedMomentsSystemConservationLaws}(\vectorial{\dualVariable})) = \sum_{\indicesVelocity = 1}^{\numberVelocities} {\kineticVelocity}_{\indicesVelocity}^{\indicesSpace}  \kineticEntropyLegendre{\indicesVelocity}, \quad \indicesSpace \in \integerInterval{1}{\spatialDimensionality},
\end{equation}
where $\vectorial{\dualVariable} = \nabla_{\vectorial{\conservedMomentsSystemConservationLaws}}\entropy$ is the conjugate variable. 
We here warn the readers about the difference between Legendre-Fenchel--transformed quantities, denoted by the symbol $\legendreTransformed{}$, and the notation employing a ${ }^{\almostLegendreTransformedSymbol}$ in \eqref{eq:transformationEntropyFlux}.

Let us introduce the choices of discrete velocities and relaxation systems that we address in the paper.

    \subsubsection{$\spatialDimensionality = 1$: One-dimensional problems}
    We consider a two-velocities model, so we set $\numberVelocities = 2$, having $\kineticVelocity_1 = \kineticVelocity > 0$ ( indexed by $+$, for the associated distribution function is transported in the positive direction) and $\kineticVelocity_2 = -\kineticVelocity < 0$ (indexed by $-$).
    The only way of fulfilling \eqref{eq:constraintsKinetic} is to select the following equilibria:
    \begin{equation*}
        \vectorial{\distributionFunctionLetter}_{\pm}^{\atEquilibrium}(\vectorial{\conservedMomentsSystemConservationLaws} ) = \frac{1}{2} \vectorial{\conservedMomentsSystemConservationLaws} \pm \frac{1}{2\kineticVelocity} \vectorial{\fluxSystemConservationLaws}(\vectorial{\conservedMomentsSystemConservationLaws}).
    \end{equation*}
    Using the change of basis $\vectorial{\conservedMomentsSystemConservationLaws} = \vectorial{\distributionFunctionLetter}_+ + \vectorial{\distributionFunctionLetter}_-$ and $\vectorial{\nonConservedMomentsSystemConservationLaws} = \kineticVelocity (\vectorial{\distributionFunctionLetter}_+ - \vectorial{\distributionFunctionLetter}_-)$, \eqref{eq:relaxationKineticSystem} can be recast as 
    \begin{equation}\label{eq:relaxationSystem1D}
        \begin{cases}
            \partial_{\timeVariable} \vectorial{\conservedMomentsSystemConservationLaws} + \partial_{\spaceVariable} \vectorial{\nonConservedMomentsSystemConservationLaws} = 0, \\
            \partial_{\timeVariable} \vectorial{\nonConservedMomentsSystemConservationLaws} + \kineticVelocity^2 \partial_{\spaceVariable} \vectorial{\conservedMomentsSystemConservationLaws} = -\frac{1}{\relaxationTime}(\vectorial{\nonConservedMomentsSystemConservationLaws} - \vectorial{\fluxSystemConservationLaws}(\vectorial{\conservedMomentsSystemConservationLaws})),
        \end{cases}
    \end{equation}
    being the well-known Jin-Xin relaxation system \cite{jin1995relaxation}.
    In the lattice Boltzmann nomenclature, this is a \lbmSchemeVectorial{1}{2}{\numberConservationLaws} relaxation system \cite{graille2014approximation}.

    Let us provide a few examples of problems that can be tackled using this scheme.

    \begin{example}[Linear transport equation]\label{eq:linearTransportEquation}
        Let $\numberConservationLaws = 1$ and $\fluxSystemConservationLaws(\conservedMomentsSystemConservationLaws) = \transportVelocity\conservedMomentsSystemConservationLaws$.
        We consider the classic quadratic entropy given by $\entropy(\conservedMomentsSystemConservationLaws) = \frac{1}{2}\conservedMomentsSystemConservationLaws^2$, thus the entropy flux is given by $\entropyFlux(\conservedMomentsSystemConservationLaws) = \frac{\transportVelocity}{2} \conservedMomentsSystemConservationLaws^2$.
        We obtain $\legendreTransformed{\entropy}(\dualVariable)= \frac{1}{2} \dualVariable^2$ and $\almostLegendreTransformed{\entropyFlux}(\dualVariable)= \frac{\transportVelocity}{2} \dualVariable^2$ (observe that $\legendreTransformed{\entropyFlux}(\dualVariable)= \frac{1}{2\transportVelocity} \dualVariable^2$).
        The dual kinetic entropies satisfying the imposed constraints are 
        \begin{equation*}
            \kineticEntropyLegendre{\pm}(\dualVariable) = \frac{1}{4} \Bigl ( 1 \pm \frac{\transportVelocity}{\kineticVelocity}\Bigr )\dualVariable^2, \qquad \text{thus} \qquad \kineticEntropy{\pm}(\distributionFunctionLetter_{\pm}) = \frac{\kineticVelocity}{\kineticVelocity\pm \transportVelocity} (\distributionFunctionLetter_{\pm})^2.
        \end{equation*}
        The condition providing the convexity of the kinetic entropy is the  sub-characteristic condition $|\transportVelocity| < \kineticVelocity$ that is found in many works, see \cite{bouchut1999construction, aregba2000discrete}.
    \end{example}
    \begin{example}[Burger' equation]
        Let $\numberConservationLaws = 1$ and $\fluxSystemConservationLaws(\conservedMomentsSystemConservationLaws) = \tfrac{1}{2}\conservedMomentsSystemConservationLaws^2$.
        We take $\entropy(\conservedMomentsSystemConservationLaws) = \frac{1}{2}\conservedMomentsSystemConservationLaws^2$, thus the entropy flux is given by $\entropyFlux(\conservedMomentsSystemConservationLaws) = \frac{1}{3} \conservedMomentsSystemConservationLaws^3$.
        Analogously to the linear case, we obtain 
        \begin{equation*}
            \kineticEntropyLegendre{\pm}(\dualVariable) = \frac{1}{4} \dualVariable^2 \pm \frac{1}{12\kineticVelocity} \dualVariable^3, \qquad \kineticEntropy{\pm}(\distributionFunctionLetter_{\pm}) = \frac{\kineticVelocity^2}{6} \Bigl ( \Bigl (1 \pm \frac{4\distributionFunctionLetter_{\pm}}{\kineticVelocity} \Bigr )^{3/2} \mp  \frac{6\distributionFunctionLetter_{\pm}}{\kineticVelocity} - 1\Bigr ).
        \end{equation*}
        Again, convexity comes by $|\conservedMomentsSystemConservationLaws| < \kineticVelocity$.
    \end{example}
    \begin{example}[Shallow water system]
        Let $\numberConservationLaws = 2$, $\vectorial{\conservedMomentsSystemConservationLaws}={(\conservedMomentsSystemConservationLaws_1, \conservedMomentsSystemConservationLaws_2)} = {(\heightShallowWater, \heightShallowWater\velocityShallowWater)}$, and $\vectorial{\fluxSystemConservationLaws}(\heightShallowWater, \heightShallowWater\velocityShallowWater) = {(\heightShallowWater\velocityShallowWater, \heightShallowWater\velocityShallowWater^2 + \tfrac{\gravityShallowWater}{2}\heightShallowWater^2)}$ where $\gravityShallowWater > 0$ is the gravity acceleration.
        As in \cite[Section 3.2]{bouchut2004nonlinear}, we take $\entropy(\heightShallowWater, \heightShallowWater\velocityShallowWater) = \tfrac{1}{2}\heightShallowWater\velocityShallowWater^2 + \tfrac{\gravityShallowWater}{2}\heightShallowWater^2$, hence $\entropyFlux(\heightShallowWater, \heightShallowWater\velocityShallowWater) = \tfrac{1}{2}\heightShallowWater\velocityShallowWater^3 + \tfrac{\gravityShallowWater}{2}\heightShallowWater^2\velocityShallowWater$.
        We have the dual variables $\dualVariable_1 = -\tfrac{1}{2}\velocityShallowWater^2 + \gravityShallowWater \heightShallowWater$ and $\dualVariable_2 = \velocityShallowWater$ and the dual kinetic entropies given by 
        \begin{equation*}
            \kineticEntropyLegendre{\pm} (\vectorial{\dualVariable}) = \frac{(\kineticVelocity\pm \dualVariable_2)(2\dualVariable_1 + \dualVariable_2^2)^2}{16\gravityShallowWater\kineticVelocity}.
        \end{equation*}
        Convexity comes under the condition $\kineticVelocity > |\velocityShallowWater| + \sqrt{\gravityShallowWater\heightShallowWater}$, see \cite{guillon:hal-03986533}, which guarantees that the kinetic velocity is larger than the one of the fastest wave in the system.
        The kinetic entropies can be found analytically, albeit with complicated formul\ae: they are given by $\kineticEntropy{\pm}(\vectorial{\distributionFunctionLetter}_{\pm}) = \vectorial{\dualVariable}_{\pm}(\vectorial{\distributionFunctionLetter}_{\pm}) \cdot \vectorial{\distributionFunctionLetter}_{\pm} - \kineticEntropyLegendre{\pm}( \vectorial{\dualVariable}_{\pm}(\vectorial{\distributionFunctionLetter}_{\pm}))$, where $\vectorial{\dualVariable}_{\pm}(\vectorial{\distributionFunctionLetter}_{\pm})$ is the solution of $\vectorial{\distributionFunctionLetter}_{\pm} = \nabla_{\vectorial{\dualVariable}}\kineticEntropyLegendre{\pm}(\vectorial{\dualVariable}_{\pm})$.
        One can see that the first equation to solve is linear in $\dualVariable_{\pm}^1$, thus we obtain 
        \begin{equation*}
            \dualVariable_{\pm}^1(\dualVariable_{\pm}^2) = \frac{4\kineticVelocity\gravityShallowWater \distributionFunctionLetter_{\pm}^1 - \kineticVelocity (\dualVariable_{\pm}^2)^2 \mp (\dualVariable_{\pm}^2)^3}{2(\kineticVelocity \pm \dualVariable_{\pm}^2)},
        \end{equation*}
        which corresponds to a third-order equation on $\dualVariable_{\pm}^2$ only:
        \begin{equation*}
            \pm \kineticVelocity\gravityShallowWater(\distributionFunctionLetter_{\pm}^1)^2 + \distributionFunctionLetter_{\pm}^1 (\dualVariable_{\pm}^2)^3 - \kineticVelocity^2 \distributionFunctionLetter_{\pm}^2 \pm (2\kineticVelocity\distributionFunctionLetter_{\pm}^1 \mp \distributionFunctionLetter_{\pm}^2)(\dualVariable_{\pm}^2)^2 + (\kineticVelocity^2 \distributionFunctionLetter_{\pm}^1 \pm2\kineticVelocity\distributionFunctionLetter_{\pm}^2)\dualVariable_{\pm}^2 = 0.
        \end{equation*}
        This equation can be solved with the well-known formula for cubic equations, upon choosing a specific branch.
    \end{example}
    \subsubsection{$\spatialDimensionality = 2$: Two-dimensional problems}
    We consider a four-velocities model, so we set $\numberVelocities = 4$, having $\vectorial{\kineticVelocity}_1 = {(\kineticVelocity, 0)}$ (indexed by $+,x$) with $\kineticVelocity > 0$, $\vectorial{\kineticVelocity}_2 = {(0, \kineticVelocity)}$ (indexed by $+,y$), $\vectorial{\kineticVelocity}_3 = {(-\kineticVelocity, 0)}$ (indexed by $-,x$), and $\vectorial{\kineticVelocity}_4 = {(0, -\kineticVelocity)}$ (indexed by $-,y$).
    There are several ways of enforcing \eqref{eq:constraintsKinetic}: the one we select is \cite[Chapter 3]{fevrier:tel-01126994}
    \begin{equation*}
        \vectorial{\distributionFunctionLetter}^{\atEquilibrium}_{\pm, x/y}(\vectorial{\conservedMomentsSystemConservationLaws} ) = \frac{1}{4} \vectorial{\conservedMomentsSystemConservationLaws} \pm \frac{1}{2\kineticVelocity} \vectorial{\fluxSystemConservationLaws}_{x/y}(\vectorial{\conservedMomentsSystemConservationLaws}).
    \end{equation*}
    Using the change of basis $\vectorial{\conservedMomentsSystemConservationLaws} = \vectorial{\distributionFunctionLetter}_{+, x} + \vectorial{\distributionFunctionLetter}_{+, y} + \vectorial{\distributionFunctionLetter}_{-, x} + \vectorial{\distributionFunctionLetter}_{-, y}$, $\vectorial{\nonConservedMomentsSystemConservationLaws}_x = \kineticVelocity (\vectorial{\distributionFunctionLetter}_{+, x} - \vectorial{\distributionFunctionLetter}_{-, x} )$, $\vectorial{\nonConservedMomentsSystemConservationLaws}_y = \kineticVelocity (\vectorial{\distributionFunctionLetter}_{+, y} - \vectorial{\distributionFunctionLetter}_{-, y} )$, and $\vectorial{\energyNonConservedMomentsSystemConservationLaws} = \kineticVelocity^2(\vectorial{\distributionFunctionLetter}_{+, x} - \vectorial{\distributionFunctionLetter}_{+, y} + \vectorial{\distributionFunctionLetter}_{-, x} - \vectorial{\distributionFunctionLetter}_{-, y})$, we get
    \begin{equation}\label{eq:relaxationSystem2D}
        \begin{cases}
            \partial_{\timeVariable} \vectorial{\conservedMomentsSystemConservationLaws} + \partial_{x} \vectorial{\nonConservedMomentsSystemConservationLaws}_x + \partial_{y} \vectorial{\nonConservedMomentsSystemConservationLaws}_y = 0, \\
            \partial_{\timeVariable} \vectorial{\nonConservedMomentsSystemConservationLaws}_x + \kineticVelocity^2 \partial_{x} \bigl (\tfrac{1}{2} \vectorial{\conservedMomentsSystemConservationLaws} + \tfrac{1}{2\kineticVelocity^2} \vectorial{\energyNonConservedMomentsSystemConservationLaws}\bigr )= -\frac{1}{\relaxationTime}(\vectorial{\nonConservedMomentsSystemConservationLaws}_x - \vectorial{\fluxSystemConservationLaws}_x(\vectorial{\conservedMomentsSystemConservationLaws})), \\
            \partial_{\timeVariable} \vectorial{\nonConservedMomentsSystemConservationLaws}_y + \kineticVelocity^2 \partial_{y} \bigl (\tfrac{1}{2} \vectorial{\conservedMomentsSystemConservationLaws} - \tfrac{1}{2\kineticVelocity^2} \vectorial{\energyNonConservedMomentsSystemConservationLaws}\bigr )= -\frac{1}{\relaxationTime}(\vectorial{\nonConservedMomentsSystemConservationLaws}_y -  \vectorial{\fluxSystemConservationLaws}_y(\vectorial{\conservedMomentsSystemConservationLaws})), \\
            \partial_{\timeVariable} \vectorial{\energyNonConservedMomentsSystemConservationLaws} + \kineticVelocity^2 \partial_{x} \vectorial{\nonConservedMomentsSystemConservationLaws}_x - \kineticVelocity^2 \partial_{y} \vectorial{\nonConservedMomentsSystemConservationLaws}_y = -\frac{1}{\relaxationTime}\vectorial{\energyNonConservedMomentsSystemConservationLaws}.
        \end{cases}
    \end{equation}
    This can be called a \lbmSchemeVectorial{2}{4}{\numberConservationLaws} relaxation system \cite{dubois2014simulation}.
    \begin{example}[Linear transport equation]
        Let $\numberConservationLaws = 1$ and ${\fluxSystemConservationLaws}^1(\conservedMomentsSystemConservationLaws) = (\transportVelocity_x\conservedMomentsSystemConservationLaws)$, ${\fluxSystemConservationLaws}^2(\conservedMomentsSystemConservationLaws) = \transportVelocity_y\conservedMomentsSystemConservationLaws$.
        We consider $\entropy(\conservedMomentsSystemConservationLaws) = \frac{1}{2}\conservedMomentsSystemConservationLaws^2$, thus the entropy flux is given by ${\entropyFlux}(\conservedMomentsSystemConservationLaws) = (\frac{\transportVelocity_x}{2} \conservedMomentsSystemConservationLaws^2, \frac{\transportVelocity_y}{2} \conservedMomentsSystemConservationLaws^2)$.
        We obtain $\legendreTransformed{\entropy}(\dualVariable)= \frac{1}{2} \dualVariable^2$ and $\almostLegendreTransformed{{\entropyFlux}}(\dualVariable)= (\frac{\transportVelocity_x}{2} \dualVariable^2, \frac{\transportVelocity_y}{2} \dualVariable^2)$.
        Possible dual kinetic entropies satisfying the constraints are 
        \begin{equation*}
            \kineticEntropyLegendre{\pm, x/y}(\dualVariable) = \frac{1}{4} \Bigl ( \frac{1}{2} \pm \frac{\transportVelocity_{x/y}}{\kineticVelocity}\Bigr )\dualVariable^2, \qquad \text{thus} \qquad \kineticEntropy{\pm, x/y}(\distributionFunctionLetter_{\pm, x/y}) = \frac{2\kineticVelocity}{\kineticVelocity\pm 2\transportVelocity_{x/y}} (\distributionFunctionLetter_{\pm, x/y})^2.
        \end{equation*}
        The conditions providing the convexity of the kinetic entropies read $|\transportVelocity_{x/y}| < \kineticVelocity/2$, see \cite{guillon:hal-03986533}.
    \end{example}

Besides the specific choices of discrete velocities that we have presented hitherto, the techniques developed in the paper work as long as $\vectorial{\kineticVelocity}_{\indicesVelocity} \in \kineticVelocity \relatives^{\spatialDimensionality}$ for $\indicesVelocity \in \integerInterval{1}{\numberVelocities}$ with a given $\kineticVelocity \in \reals$.
For example, one could employ the well-known \lbmSchemeVectorial{2}{9}{} scheme \cite{lallemand2000theory}.

\newcommand{\strong}[1]{\textbf{#1}}

\section{Numerical schemes}\label{sec:kineticSchemes}

Now that we have set the preliminaries concerning relaxation systems at a continuous level, we are ready to propose several numerical schemes to tackle \eqref{eq:systemConservationLaws} inspired by \eqref{eq:relaxationKineticSystem}.

The first step is---in \Cref{sec:standardLBM}---the introduction of transport and relaxation phases, yielding the standard lattice Boltzmann scheme. 
This scheme can be easily made second-order accurate; however, it is hard to push it towards higher accuracy because the scheme lacks time-symmetry. The time-symmetry property is indeed useful for increasing the order of the scheme through palindromic composition \cite{mclachlan2002splitting}.
The second step in the process, presented in \Cref{sec:symmetrization}, is conducted by symmetrization, without increasing the actual order of the scheme.
The latter is the aim of the third step in the process and is obtained by composition, as detailed in \Cref{sec:fourthOrderSchemes}.

For the space discretization, we employ a uniform Cartesian mesh $\spaceLattice{\spatialDimensionality}$---also known as lattice---of step $\spaceStep > 0$.
The uniform time step is denoted by $\timeStep > 0$ and is specified in what follows.

\subsection{Standard lattice Boltzmann schemes}\label{sec:standardLBM}

The left-hand side of \eqref{eq:relaxationKineticSystem} is made up of linear transport equations with constant velocities $\vectorial{\kineticVelocity}_{\indicesVelocity}$, whereas the right-hand side represents local relaxations.
It is therefore natural to split these two terms and let them undergo different treatments.
In this part of the paper, we consider $\indicesVelocity \in \integerInterval{1}{\numberVelocities}$ be the index of any discrete velocity.

\subsubsection{Transport}
The equations associated with the left-hand side of \eqref{eq:relaxationKineticSystem} are solved using any consistent one-step scheme for the linear transport equation.
Recall that the kinetic velocities $\vectorial{\kineticVelocity}_{\indicesVelocity}$ are integer multiples of $\kineticVelocity$, which is adjusted so that the kinetic velocity $\kineticVelocity$ fulfills $\kineticVelocity \timeStep/\spaceStep = \numberShiftsTransport \in \naturalsWithoutZero$.
In this way, the transport phase is indeed given by elementary shifts on the grid, which is the natural issue of any consistent one-step scheme in this peculiar framework.
The fact of shifting data sticking to the discrete grid  makes our approach a lattice Boltzmann approach.
This reads 
\begin{equation}\label{eq:transportPhase}
    \vectorial{\distributionFunctionLetter}_{\indicesVelocity}(\timeStep, \vectorial{\spaceVariable}) = \vectorial{\distributionFunctionLetter}_{\indicesVelocity}(0, \vectorial{\spaceVariable} - \vectorial{\kineticVelocity}_{\indicesVelocity}\timeStep) = \vectorial{\distributionFunctionLetter}_{\indicesVelocity}(0, \vectorial{\spaceVariable} - \underbrace{\numberShiftsTransport \tfrac{\vectorial{\kineticVelocity}_{\indicesVelocity}}{\kineticVelocity}}_{\in \relatives^{\spatialDimensionality}}\spaceStep), \qquad \vectorial{\spaceVariable} \in \spaceLattice{\spatialDimensionality}.
\end{equation}
Gathering all the distribution functions together, this transport phase is denoted by $\transport(\timeStep)$.
This operator is made up of exact schemes for the transport equations; however, one must be aware that we have performed an overall splitting between transport and relaxation, thus this does not ensure accuracy with respect to the original problem \eqref{eq:relaxationKineticSystem}---and \emph{a fortiori} with \eqref{eq:systemConservationLaws}---above first-order.

\subsubsection{Relaxation}

The relaxation part, \idEst{} the right-hand side of \eqref{eq:relaxationKineticSystem}, is solved using a trapezoidal quadrature, see \cite{dellar2013interpretation}, which is second-order accurate.
Using the fact that the relaxation phase conserves $\vectorial{\conservedMomentsSystemConservationLaws}$ and thus any equilibrium fulfills $\vectorial{\distributionFunctionLetter}_{\indicesVelocity}^{ \atEquilibrium}(\vectorial{\conservedMomentsSystemConservationLaws}(\splittingTimeStep) ) = \vectorial{\distributionFunctionLetter}_{\indicesVelocity}^{\atEquilibrium}(\vectorial{\conservedMomentsSystemConservationLaws} (0))$, the algorithm can be fortunately kept explicit and thus reads
\begin{equation}\label{eq:relaxationDiscretization}
    \vectorial{\distributionFunctionLetter}_{\indicesVelocity}(\splittingTimeStep) = \frac{2\relaxationTime - \splittingTimeStep}{2\relaxationTime + \splittingTimeStep} \vectorial{\distributionFunctionLetter}_{\indicesVelocity}(0) + \frac{2\splittingTimeStep}{2\relaxationTime + \splittingTimeStep} \vectorial{\distributionFunctionLetter}^{ \atEquilibrium}_{\indicesVelocity}(\vectorial{\conservedMomentsSystemConservationLaws} (0)) \xrightarrow[\relaxationTime \to 0]{} - \vectorial{\distributionFunctionLetter}_{\indicesVelocity}(0) + 2 \vectorial{\distributionFunctionLetter}_{\indicesVelocity}^{ \atEquilibrium}(\vectorial{\conservedMomentsSystemConservationLaws} (0)).
\end{equation}
The space variable is not listed since the relaxation step is local and performed at each point of the spatial grid $\spaceLattice{\spatialDimensionality}$.
We consider the limit $\relaxationTime \to 0$ in \eqref{eq:relaxationDiscretization}, thus a relaxation independent of $\splittingTimeStep$.
Therefore, $\relaxationTime$ no longer appears in the numerical scheme. The relaxation system and its discretization must be seen as an intermediate step to propose a numerical scheme for another equation where no relaxation time exists, namely \eqref{eq:systemConservationLaws}.
More generally, a relaxation independent of $\timeStep$ can be written as 
\begin{equation}\label{eq:relaxationDiscretizationOmega}
    \vectorial{\distributionFunctionLetter}_{\indicesVelocity}(\splittingTimeStep) =  (1-\relaxationParameter) \vectorial{\distributionFunctionLetter}_{\indicesVelocity}(0) + \relaxationParameter \vectorial{\distributionFunctionLetter}_{\indicesVelocity}^{ \atEquilibrium}(\vectorial{\conservedMomentsSystemConservationLaws} (0)).
\end{equation}
with a relaxation parameter $\relaxationParameter \in (0, 2]$, and we indicate it by $\relaxation{\relaxationParameter}$.
Whenever we write $\relaxation{}$, we mean $\relaxation{\relaxationParameter = 2}$.
Notice that $\relaxation{\relaxationParameter =2}$ is an involution: $\relaxation{\relaxationParameter =2} \relaxation{\relaxationParameter =2} = \identityOperator$, which is false for relaxation parameters $\relaxationParameter < 2$.
This feature of the relaxation operator is crucial in what follows.

\subsubsection{Overall lattice Boltzmann scheme}

One can show that the scheme $\basicBrickSolver(\splittingTimeStep) = \relaxation{} \transport(\timeStep)$ (or $\basicBrickSolver(\splittingTimeStep) = \transport(\timeStep) \relaxation{}$) is a second-order scheme to solve \eqref{eq:systemConservationLaws}.
This boils down to the standard SRT (Single-Relaxation-Time) lattice Boltzmann scheme with relaxation parameter equal to two \cite{graille2014approximation}, which is second-order accurate, see \cite{dubois2022nonlinear, bellotti2023truncation}.

However, these schemes are not time-symmetric.
Time symmetry is defined by 
\begin{equation}\label{eq:symmetryCondition}
    \basicBrickSolver(\splittingTimeStep) \basicBrickSolver(-\splittingTimeStep) = \identityOperator \qquad \textnormal{and} \qquad \basicBrickSolver(0) = \identityOperator.
\end{equation} 
In the present case, $\basicBrickSolver(\splittingTimeStep)\basicBrickSolver(-\splittingTimeStep) \neq \identityOperator$, where whenever we employ negative time-steps, it is like if we simply reverse $\kineticVelocity \mapsto -\kineticVelocity$ using a positive time-step, \idEst{} the distribution functions are transported in the opposite direction compared to what they would do when $\kineticVelocity > 0$.
Time-symmetry is a highly desirable feature that fosters the increase of the order by using composition procedures, in the spirit of what \cite{mclachlan2002splitting} presents.
We now try to fix this problem.

\subsection{Symmetric lattice Boltzmann schemes}\label{sec:symmetrization}

The first idea is to use a sort of Strang formula that would read $\basicBrickSolver(\splittingTimeStep) = \transport  ( \frac{\splittingTimeStep}{2} ) \relaxation{} \transport  ( \frac{\splittingTimeStep}{2} )$.
Since the transport phase is made up of elementary shifts on the grid, we have that $\basicBrickSolver(\splittingTimeStep) \basicBrickSolver(-\splittingTimeStep) =\identityOperator$.
However, $\basicBrickSolver(0) \neq \identityOperator$, thus this operator is not suitable to be employed to increase the overall order of the numerical scheme.
Juxtaposing two half-steps of this operator---following \cite{coulette2019high}---we can take advantage of the involution property of the relaxation operator $\relaxation{}$, and gain 
\begin{equation}\label{eq:basicBrickSplitting}
    \basicBrickSolver(\splittingTimeStep) = \transport \Bigl ( \frac{\splittingTimeStep}{4}\Bigr ) \relaxation{} \transport \Bigl ( \frac{\splittingTimeStep}{4}\Bigr ) \transport \Bigl ( \frac{\splittingTimeStep}{4}\Bigr ) \relaxation{} \transport \Bigl ( \frac{\splittingTimeStep}{4}\Bigr ) =  \transport \Bigl ( \frac{\splittingTimeStep}{4}\Bigr ) \relaxation{} \transport \Bigl ( \frac{\splittingTimeStep}{2}\Bigr ) \relaxation{} \transport \Bigl ( \frac{\splittingTimeStep}{4}\Bigr ).
\end{equation}
One can easily check that, thanks to the fact that $\relaxation{}$ is an involution, we have \eqref{eq:symmetryCondition} hence $\basicBrickSolver$ defined through \eqref{eq:basicBrickSplitting} is time-symmetric.
So far, nothing special has been done to increase the order, so \eqref{eq:basicBrickSplitting} is just another second-order accurate solver, as the one provided by the  Strang formula $ \transport  ( \frac{\splittingTimeStep}{2} ) \relaxation{} \transport  ( \frac{\splittingTimeStep}{2} )$.

\subsection{Fourth-order lattice Boltzmann scheme}\label{sec:fourthOrderSchemes}

The symmetry property is crucial to obtain high-order schemes by composition. 
Let us assume that the symmetric lattice Boltzmann operator $\basicBrickSolver(\splittingTimeStep)$ defined in \eqref{eq:basicBrickSplitting} 
leads to a converging lattice Boltzmann scheme. In other words, this means that $\basicBrickSolver(\splittingTimeStep)$ is an approximation of the flow of a formal differential equation $(\vectorial{\distributionFunctionLetter}_1, \dots, \vectorial{\distributionFunctionLetter}_{\numberVelocities})'(t, \cdot) = \vectorial{\abstractProblemFunction}((\vectorial{\distributionFunctionLetter}_1, \dots, \vectorial{\distributionFunctionLetter}_{\numberVelocities})(t, \cdot))$.
In the Lie groups theory, \confer{} \cite[Introduction]{mclachlan2002splitting} and \cite[Chapter 2 and 3]{hall2013lie}, it is common to denote the flow of the differential equation by the exponential notation $(\vectorial{\distributionFunctionLetter}_1, \dots, \vectorial{\distributionFunctionLetter}_{\numberVelocities})(t, \cdot) = e^{\timeVariable\vectorial{\abstractProblemFunction}} ((\vectorial{\distributionFunctionLetter}_1, \dots, \vectorial{\distributionFunctionLetter}_{\numberVelocities})(0, \cdot))$. This is a generalization of the matrix exponential from the linear case.
With this, we have $\basicBrickSolver(\splittingTimeStep) \approx e^{\splittingTimeStep\vectorial{\abstractProblemFunction}}$.
In \cite{drui2019analysis, guillon:hal-03986533}, it is shown that $\sum_{\indicesVelocity = 1}^{\indicesVelocity=\numberVelocities} \vectorial{\abstractProblemFunction}_{\indicesVelocity} (\vectorial{\distributionFunctionLetter}_1, \dots, \vectorial{\distributionFunctionLetter}_{\numberVelocities})$ depends only on $\vectorial{\conservedMomentsSystemConservationLaws} = \sum_{\indicesVelocity = 1}^{\indicesVelocity = \numberVelocities} \vectorial{\distributionFunctionLetter}_{\indicesVelocity}$, and
\begin{equation*}
    \sum_{\indicesVelocity = 1}^{\numberVelocities} \vectorial{\abstractProblemFunction}_{\indicesVelocity}(\vectorial{\conservedMomentsSystemConservationLaws}) = - \sum_{\indicesSpace = 1}^{\spatialDimensionality} \frac{\partial}{\partial \spaceVariable_{\indicesSpace}}  \vectorial{\fluxSystemConservationLaws}^{\indicesSpace}(\vectorial{\conservedMomentsSystemConservationLaws}),
\end{equation*}
so that the scheme eventually solves \eqref{eq:systemConservationLaws}, plus other equations that can be made explicit.
It is possible to be more precise: by time-symmetry, using \cite[Theorem 19]{mclachlan2002splitting}, there exists a vector field $\vectorial{\perturbationVectorField}$ such that
\begin{equation*}
    \basicBrickSolver(\timeStep) = e^{\timeStep \vectorial{\abstractProblemFunction} + \timeStep^3 \vectorial{\perturbationVectorField}} + \bigO{\timeStep^5},
\end{equation*}
indicating that $\basicBrickSolver$ is second-order accurate.
Remark that there is no guarantee on the fact that $\vectorial{\abstractProblemFunction}$ and $\vectorial{\perturbationVectorField}$ commute.
Following \cite[Equation (4.4)]{mclachlan2002splitting}, we look for an overall operator---constructed by composition---under the form 
\begin{equation}\label{eq:splitting}
    \globalScheme(\timeStep) = \basicBrickSolver(\coefficientOuterStep\timeStep)^{\numberOuterSteps} \basicBrickSolver(\coefficientInnerStep\timeStep) \basicBrickSolver(\coefficientOuterStep\timeStep)^{\numberOuterSteps},
\end{equation}
where $\numberOuterSteps \in \naturalsWithoutZero$.
According to \cite[Theorem 22]{mclachlan2002splitting}, the operator $\globalScheme$ is such that 
\begin{equation*}
    \globalScheme(\timeStep) = e^{\timeStep (2 \numberOuterSteps \coefficientOuterStep + \coefficientInnerStep) \vectorial{\abstractProblemFunction} + \timeStep^3 (2 \numberOuterSteps \coefficientOuterStep^3 + \coefficientInnerStep^3) \vectorial{\perturbationVectorField}} + \bigO{\timeStep^5},
\end{equation*}
thus it has a local truncation error of order five---thus it is globally accurate at order four---provided that the conditions 
\begin{align}
    2\numberOuterSteps \coefficientOuterStep + \coefficientInnerStep &= 1, \label{eq:fourtOrderConditionOne}\\
    2 \numberOuterSteps \coefficientOuterStep^3 + \coefficientInnerStep^3 &= 0,\label{eq:fourtOrderConditionTwo}
\end{align}
are satisfied.
For we want to deal with a lattice Boltzmann approach, characterized by the fact that the transport phase \eqref{eq:transportPhase} is made up of integer shifts on the discrete spatial grid $\spaceLattice{\spatialDimensionality}$, we would like $\coefficientOuterStep, \coefficientInnerStep \in \rationals$.\footnote{It could still be possible to obtain a lattice Boltzmann scheme whenever $\coefficientOuterStep, \coefficientInnerStep \in \reals \smallsetminus \rationals$, provided that $\coefficientOuterStep$ and $\coefficientInnerStep$ are commensurable.}
In order to fulfill \eqref{eq:fourtOrderConditionTwo}, one can easily see that either $\coefficientOuterStep$ or $\coefficientInnerStep$ has to be negative, meaning that steps with transport according to the sign of the discrete velocities are interspersed with steps in the opposite direction.
Otherwise said, the price to pay to obtain fourth-order consistency is to alternate steps both forward and backward in time.
Inserting \eqref{eq:fourtOrderConditionOne} into \eqref{eq:fourtOrderConditionTwo} gives $2 \numberOuterSteps \coefficientOuterStep^3 + (1-2\numberOuterSteps \coefficientOuterStep)^3 = 0$.
For $\numberOuterSteps = 1$ the only real solution is  irrational.
The same holds for $\numberOuterSteps = 2, 3$, and these cases are not of interest in our setting, because once put back into \eqref{eq:fourtOrderConditionOne}, both  $\coefficientOuterStep$ and $\coefficientInnerStep$ are irrational and eventually incommensurable.
For $\numberOuterSteps = 4$, we have the rational solution $\coefficientOuterStep = 1/6$, hence $\coefficientInnerStep = -1/3$.
Therefore, the formula that we retain is 
\begin{equation}\label{eq:splittingToUse}
    \globalScheme(\timeStep) = \basicBrickSolver\Bigl (\frac{\timeStep}{6}\Bigr )^4\basicBrickSolver\Bigl (-\frac{\timeStep}{3} \Bigr ) \basicBrickSolver\Bigl (\frac{\timeStep}{6}\Bigr )^4.
\end{equation}
Looking at \eqref{eq:splittingToUse} and \eqref{eq:basicBrickSplitting}, we see that the shortest transport phase features a time-step equal to $\timeStep/24$.
This means that ``particles'' roughly travel $\kineticVelocity \timeStep/(24\spaceStep)$ gridpoints at each time the transport operator is called. 
To ensure that the scheme remains a lattice Boltzmann scheme, we enforce that $\kineticVelocity \timeStep/(24\spaceStep) = \numberShiftsTransport \in \naturalsWithoutZero$.
The time step is given by $\timeStep = 24 \numberShiftsTransport \spaceStep / \kineticVelocity$.
We consistently take $\numberShiftsTransport = 1$.
The kinetic velocity is freely chosen, still ensuring that all the waves in \eqref{eq:systemConservationLaws} are resolved:
\begin{equation}\label{eq:choiceKineticVelocity}
    \kineticVelocity \greatersim \max_{i \in \integerInterval{1}{\numberConservationLaws}} |\lambda_i|,
\end{equation}
where the $\lambda_i$ are the eigenvalues of the Jacobian matrix of $\vectorial{\fluxSystemConservationLaws}$, \idEst{} the velocities of the waves.
The kinetic velocity should not be too large compared to the fastest wave in the system, to ensure accuracy.

\begin{remark}
    We see that according to \eqref{eq:splittingToUse}, at each time step, ``particles'' undergo $1\times 4 \times 4 \times \numberShiftsTransport = 16\numberShiftsTransport$ shifts according to the sign of their velocities and eight relaxations, followed by $2\times 1 \times 4 \times \numberShiftsTransport = 8\numberShiftsTransport$ shifts (of twice the length) in the opposite direction and two relaxations, followed again by $1\times 4 \times 4 \times \numberShiftsTransport = 16\numberShiftsTransport$ shifts according to the sign of their velocities and eight relaxations.
\end{remark}

\begin{remark}
    By making the change of variable $\timeStep \mapsto 24\timeStep$, we can interpret things in another manner.
    The overall scheme is fourth-order accurate if we observe it every $24$ time steps doing 
    \begin{multline*}
        \overbrace{\transport(\timeStep)\relaxation{}\transport(2\timeStep)\relaxation{} \transport(\timeStep) \times \cdots \times \transport(\timeStep)\relaxation{}\transport(2\timeStep)\relaxation{} \transport(\timeStep)}^{4 \text{ times}} \\
        \times \transport(-2\timeStep)\relaxation{}\transport(-4\timeStep)\relaxation{} \transport(-2\timeStep)  \\
        \times \underbrace{\transport(\timeStep)\relaxation{}\transport(2\timeStep)\relaxation{} \transport(\timeStep) \times \cdots \times \transport(\timeStep)\relaxation{}\transport(2\timeStep)\relaxation{} \transport(\timeStep)}_{4 \text{ times}},
    \end{multline*} 
    which means having done $32$ steps (of length $\numberShiftsTransport$) forward and $8$ steps backward  (of twice the step) with a specific interleaving.
\end{remark}

\begin{remark}[Cost of the scheme]
    Considering \eqref{eq:splittingToUse} and the previous remark, the cost of the whole algorithm might seem very high. 
    However, the time-marching procedure is totally made up of traditional transport and relaxation steps of a standard lattice Boltzmann method, and techniques to parallelize and deploy them of modern architectures (\emph{e.g.} GPUs) are available and indeed employed.
    A crucial advantage of the splitting strategy is that it does not require additional storage, as it is the case with a Runge-Kutta approach. 
    Moreover, the numerical solution in the inner sub-steps is consistent---\idEst{} meaningful---and simply second-order accurate. We can thus consider the method as an (almost) standard second-order lattice Boltzmann scheme, where fourth-order accuracy is observed at specific steps of the time-marching procedure.
\end{remark}


\subsection{$L^2$ stability}\label{sec:L2Stability}

We see in \Cref{sec:entropyStability} that---with a simple procedure acting on the relaxation parameter---our scheme can possess excellent features concerning entropy stability, ensuring stability in a non-linear framework.
This comes from the fact that entropy provides---reminding us of the work by \cite{junk2009weighted}---the right weighted norm to take the effect of the relaxation into account.
It is more involved to study the stability with respect to the $L^2$ norm, which furthermore applies only to a linear setting.
Even for the standard \lbmSchemeVectorial{2}{4}{} from \Cref{sec:standardLBM}, no explicit $L^2$ stability condition is known, to the best of our knowledge.

Nevertheless, we start by a brief study concerning $L^2$ stability.
We consider the case of $\spatialDimensionality = 1$ with one conservation law $\numberConservationLaws = 1$, thus we use a \lbmSchemeVectorial{1}{2}{} scheme.
Moreover, we consider a linear problem: $\fluxSystemConservationLaws(\conservedMomentsSystemConservationLaws) = \transportVelocity \conservedMomentsSystemConservationLaws$.
A polynomial with complex coefficients is said to be a simple \emph{von Neumann} polynomial if its roots are in the closed unit disk and those on the unit circle are simple.
Then, the corresponding Finite Difference scheme computed using the characteristic polynomial of $\globalScheme$ \cite{bellotti2023truncation} is $L^2$ stable if the characteristic polynomial, upon considering its Fourier transform, is a simple \emph{von Neumann} polynomial for every frequency.
Conversely, the original lattice Boltzmann scheme $\globalScheme$ is $L^2$ stable if its minimal polynomial is a simple \emph{von Neumann} polynomial for every frequency, see \cite{bellotti2023influence}.
It is well-know \cite{graille2014approximation} that the standard lattice Boltzmann scheme ($\relaxation{} \transport(\timeStep)$ or $ \transport(\timeStep) \relaxation{}$) from \Cref{sec:standardLBM} is stable for the $L^2$ norm under the strict condition 
\begin{equation}\label{eq:CFLLeapFrog}
    \frac{|\transportVelocity|\timeStep}{\numberShiftsTransport \spaceStep} < 1,
\end{equation}
which is the CFL condition of a leap-frog scheme, \confer{} \cite{strikwerda2004finite}.
For the new scheme $\globalScheme(\timeStep)$ given by \eqref{eq:splittingToUse}, we cannot conclude that it is stable provided that $\basicBrickSolver (\frac{\timeStep}{6} )$ and $\basicBrickSolver (-\frac{\timeStep}{3} )$ are stable, because these operators are not simultaneously diagonalizable, for they do not commute.
We have to study the eigenvalues of $\globalScheme(\timeStep)$.
In particular, we focus on its characteristic polynomial---whose roots include those of the minimal polynomial---as long as it allows concluding.
Otherwise, we switch to the minimal polynomial.
The characteristic polynomial reads 
\begin{equation}\label{eq:charEquation}
    \determinant(\timeShiftOperator\identityOperator - \fourierTransformed{\globalScheme}(\timeStep)(\frequency\spaceStep)) = \timeShiftOperator^2 - \trace (\fourierTransformed{\globalScheme}(\timeStep)(\frequency\spaceStep))\timeShiftOperator + \determinant(\fourierTransformed{\globalScheme}(\timeStep)(\frequency\spaceStep)),
\end{equation}
for $|\frequency\spaceStep| \leq \pi$. Here, hats denote Fourier-transformed quantities.
Since $\determinant(\fourierTransformed{\transport}(\timeStep)(\frequency\spaceStep)) = 1$, $\determinant(\fourierTransformed{\relaxation{}}(\frequency\spaceStep)) = -1$, and $\globalScheme$ is made up of an even number of relaxations $\relaxation{}$, we can use the formula for the determinant of a product of matrices and thus obtain $\determinant(\fourierTransformed{\globalScheme}(\timeStep)(\frequency\spaceStep)) = 1$.
For the trace appearing in \eqref{eq:charEquation}, less can be said. Its explicit expression---computed using a computer algebra system---is involved and provided in \Cref{sec:traceD1Q2} for the interested readers. Yet, we observe that $\trace (\fourierTransformed{\globalScheme}(\timeStep)(\frequency\spaceStep))\in \reals$.
Using the results by \cite{miller1971location}, the characteristic polynomial is a simple \emph{von Neumann} polynomial if and only if the sole root of $\tfrac{\text{d}}{\text{d}\timeShiftOperator}\determinant(\timeShiftOperator\identityOperator - \fourierTransformed{\globalScheme}(\timeStep)(\frequency\spaceStep))$ is in the open unit disk. This reads
\begin{equation}\label{eq:traceInequality}
    |\trace (\fourierTransformed{\globalScheme}(\timeStep)(\frequency\spaceStep))| < 2.
\end{equation}
One can be easily persuaded, \confer{} \Cref{sec:traceD1Q2}, that \eqref{eq:traceInequality} if fulfilled as long as $|\transportVelocity|/\kineticVelocity < 1$ (except at $\frequency\spaceStep = 0, \tfrac{\pi}{2\numberShiftsTransport}$ which need to be analyzed separately) and, when going beyond this value, that the left-hand side of \eqref{eq:traceInequality} is critically maximal at $\frequency\spaceStep = \frac{\pi}{4\numberShiftsTransport}$.
Evaluating \eqref{eq:traceInequality} at this value gives an inequality of degree 16 in $\transportVelocity/\kineticVelocity$, which, solved using \texttt{sage-math}, exactly gives \eqref{eq:stabilityConditionSousPas}.
This allows to conclude that under this condition, except for $\frequency\spaceStep = 0, \tfrac{\pi}{2\numberShiftsTransport}$, the characteristic polynomial of $\fourierTransformed{\globalScheme}(\timeStep)(\frequency\spaceStep)$ and thus the minimal polynomial are simple \emph{von Neumann}.
The two exceptional cases are exactly those where there is a gap between characteristic and minimal polynomials \cite{bellotti2023influence}. Indeed $\determinant(\timeShiftOperator\identityOperator - \fourierTransformed{\globalScheme}(\timeStep)(\frequency\spaceStep))|_{\frequency\spaceStep=0, {\pi}/{(2\numberShiftsTransport)}} = (\timeShiftOperator - 1)^2$, whereas $\fourierTransformed{\globalScheme}(\timeStep)(\frequency\spaceStep)|_{\frequency\spaceStep=0, {\pi}/{(2\numberShiftsTransport)}} = \identityOperator$ indicates that $\timeShiftOperator - 1$ is the minimal polynomial in this case, and it is simple.
Finally, we observe that we cannot include the case $|\transportVelocity|/\kineticVelocity = 1$, since $\determinant(\timeShiftOperator\identityOperator - \fourierTransformed{\globalScheme}(\timeStep)(\tfrac{\pi}{4\numberShiftsTransport}))|_{|\transportVelocity|/\kineticVelocity = 1} = (\timeShiftOperator - 1)^2$ but 
\begin{equation*}
    \fourierTransformed{\globalScheme}(\timeStep)(\tfrac{\pi}{4\numberShiftsTransport})|_{|\transportVelocity|/\kineticVelocity = 1} = 
    \begin{pmatrix}
        1 & 32i \\
        0 & 1
    \end{pmatrix}.
\end{equation*}
This allows concluding on the $L^2$ stability of the lattice Boltzmann scheme, see the following result.
\begin{proposition}
    Let $\spatialDimensionality = 1$, $\numberConservationLaws = 1$, and $\fluxSystemConservationLaws(\conservedMomentsSystemConservationLaws) = \transportVelocity \conservedMomentsSystemConservationLaws$.
    Consider a \lbmSchemeVectorial{1}{2}{} scheme.
    Then $\globalScheme(\timeStep)$ is $L^2$-stable under the condition
    \begin{equation}\label{eq:stabilityConditionSousPas}
        \frac{|\transportVelocity|}{\kineticVelocity} =  \frac{|\transportVelocity| \timeStep}{24 \numberShiftsTransport \spaceStep} < 1.
    \end{equation}
\end{proposition}
Observe that there is no difference between \eqref{eq:CFLLeapFrog} and \eqref{eq:stabilityConditionSousPas}, because in between, we have made the change of variable $\timeStep \mapsto \frac{1}{24}\timeStep$.
This is indeed the sub-characteristic condition found in \Cref{eq:linearTransportEquation}.

\section{Numerical experiments: Order of the scheme}\label{sec:experimentsOrder}

We now proceed to several numerical experiments to confirm the theoretical order of the method we have devised in a non-linear context.
We consider both scalar problems and systems in one and two spatial dimensions.
All the tests have been implemented and parallelized on GPUs using \texttt{OpenCL} \cite{baty:hal-02965967}.

\subsection{Non-linear scalar problem: Burgers' equation in 1D}\label{sec:burgers}

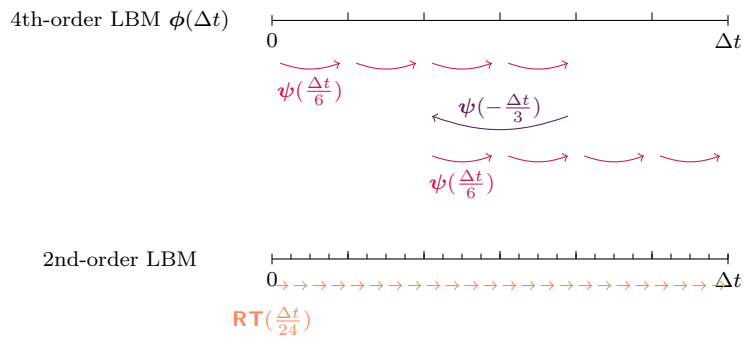
\begin{figure}[htbp]
    \begin{center}
        \begin{footnotesize}
            \begin{tikzpicture}
                \node at (-2, 0pt) {4th-order LBM $\globalScheme(\timeStep)$};

                \draw (0,0) -- (6,0);
                
                \draw (0,2pt) -- (0,-2pt) node[below] {$0$};
                \draw (6,2pt) -- (6,-2pt) node[below] {$\timeStep$};
                
                \foreach \x in {1,2,...,5}{
                  \draw (\x,0) -- (\x,3pt);
                }
                \foreach \x in {1,2,3}{
                  \node (A\x) at (\x,-15pt) {};
                  \node (B\x) at (\x+1,-15pt) {};
                  \draw[->, bend right=20, color=myRed] (A\x) to (B\x);
                }
                \node (A0) at (0,-15pt) {};
                \node (B0) at (1,-15pt) {};
                \draw[->, bend right=20, color=myRed] (A0) to node[midway, below] {$\basicBrickSolver(\tfrac{\timeStep}{6})$}  (B0);
    
                \node (C0) at (4,-35pt) {};
                \node (D0) at (2,-35pt) {};
                \draw[->, bend right=-20, color=myViolet] (C0) to node[midway, above] {$\basicBrickSolver(-\tfrac{\timeStep}{3})$}  (D0);
                  
                \foreach \x in {1,2,3}{
                  \node (E\x) at (2+\x,  -50pt) {};
                  \node (F\x) at (2+\x+1,-50pt) {};
                  \draw[->, bend right=20, color=myRed] (E\x) to (F\x);
                }
                \node (E0) at (2,-50pt) {};
                \node (F0) at (3,-50pt) {};
                \draw[->, bend right=20, color=myRed] (E0) to node[midway, below] {$\basicBrickSolver(\tfrac{\timeStep}{6})$}  (F0);
            
            \draw (0,-90pt) -- (6,-90pt);
            \node at (-2,-90pt) {2nd-order LBM};
            \draw (0,-90pt+2pt) -- (0,-90pt-2pt) node[below] {$0$};
            \draw (6,-90pt+2pt) -- (6,-90pt-2pt) node[below] {$\timeStep$};
    
            \foreach \x in {1,2,...,5}{
                  \draw (\x,-90pt) -- (\x,-90pt+3pt);
            }
            \foreach \x in {1,2,...,24}{
                  \draw (0.25*\x,-90pt) -- (0.25*\x,-90pt+2pt);
            }
            \foreach \x in {0, ..., 23}{
                  \node (A\x) at (0.25*\x,-100pt) {};
                  \node (B\x) at (0.25*\x+0.25,-100pt) {};
                  \draw[->, color=myOrange, shorten >=-2pt, shorten <=-2pt] (A\x) to (B\x);
                }
            \draw (0,-105pt) node[below, color=myOrange] {$\relaxation{} \transport(\tfrac{\timeStep}{24})$};

              \end{tikzpicture}
        \end{footnotesize}
    \end{center}\caption{\label{fig:fairComparison}Way of devising a fair comparison between our new fourth-order lattice Boltzmann scheme (top) and the original second-order lattice Boltzmann scheme (bottom).}
\end{figure}

To test the fourth-order convergence of our solver $\globalScheme$ in a genuinely non-linear setting, we consider the Burgers' equation on a bounded domain $[0, 1]$ endowed with periodic boundary conditions.
The initial datum is the point-wise discretization of $\conservedMomentsSystemConservationLaws(\timeVariable = 0, \spaceVariable) = \sin(2\pi\spaceVariable)$ with data taken at equilibrium, which means that $\vectorial{\distributionFunctionLetter}_{\indicesVelocity}(\timeVariable = 0) = \vectorial{\distributionFunctionLetter}_{\indicesVelocity}^{\atEquilibrium}(\vectorial{\conservedMomentsSystemConservationLaws}(\timeVariable = 0))$ for $\indicesVelocity \in\integerInterval{1}{\numberVelocities}$.
In order to fulfill the sub-characteristic condition, the kinetic velocity is $\kineticVelocity = 1.2$.
The final time of the simulation, at which we measure errors, is $\finalTime = 1/10$, which is before the solution exhibits a shock wave.

We would also like to compare the accuracy of our approach against the standard second-order lattice Boltzmann method given by $\relaxation{} \transport(\timeStep)$.
To ensure a fair comparison between errors at roughly the same computational cost, we have to proceed as indicated in \Cref{fig:fairComparison}, namely consider the scheme given by $(\relaxation{} \transport(\tfrac{\timeStep}{24}))^{24}$.
This allows to have roughly the same number of operations between the second-order scheme (24 steps) and the fourth-order scheme (32 steps) for advancing of $\timeStep$ in time. The additional steps are needed for the fourth-order scheme sometimes goes backward in time.

\begin{table}[htbp]
    \caption{\label{tab:convergenceBurgersOrder}Errors and order of convergence in the $L^2$ norm for the Burgers' equation using the second-order lattice Boltzmann scheme and our new fourth-order scheme.}
    \setlength{\tabcolsep}{2.5pt}
        \begin{center}
            \begin{scriptsize}
                \begin{tabular}{|c|cc|cc|}
                    \hline
                    & \multicolumn{2}{c|}{2nd-order LBM} & \multicolumn{2}{c|}{4th-order LBM} \\
                    \hline
                    $\spaceStep$ & $L^2$ error & Order & $L^2$ error & Order \\
                    \hline 
                    2.000E-03 & 8.592E-05	&     & 3.370E-06 & 	     \\ 
                    1.250E-03 & 3.358E-05	& 2.00& 1.552E-06 & 	1.65 \\ 
                    7.813E-04 & 1.404E-05	& 1.86& 1.742E-07 & 	4.65 \\ 
                    4.883E-04 & 5.494E-06	& 2.00& 3.365E-08 & 	3.50 \\ 
                    3.053E-04 & 2.160E-06	& 1.99& 5.184E-09 & 	3.98 \\ 
                    1.908E-04 & 7.799E-07	& 2.17& 8.688E-10 & 	3.80 \\ 
                    1.193E-04 & 3.057E-07	& 1.99& 1.221E-10 & 	4.18 \\ 
                    7.454E-05 & 1.287E-07	& 1.84& 2.109E-11 & 	3.74 \\ 
                    \hline 
                \end{tabular}
            \end{scriptsize}
        \end{center}
\end{table}

The results are given in \Cref{tab:convergenceBurgersOrder} and show second-order convergence for the original lattice Boltzmann scheme and fourth-order convergence for the new scheme.
Comparing the standard second-order lattice Boltzmann scheme to our new scheme in the framework where they have roughly the same computational cost, we see that even at a very coarse resolution, the fourth-order scheme outperforms the standard scheme being at least twenty times more accurate.

\subsection{Non-linear system: Shallow water equations in 1D}

\begin{table}[htbp]
    \caption{\label{tab:convergenceShalloWaters}Error estimations and order of convergence in the $L^2$ metric for the shallow water system.}
    \setlength{\tabcolsep}{2.5pt}
        \begin{center}
            \begin{scriptsize}
                \begin{tabular}{|c|cc|cc|}
                    \hline
                    & \multicolumn{2}{c|}{Height $\heightShallowWater$} & \multicolumn{2}{c|}{Velocity $\velocityShallowWater$} \\
                    \hline
                    $\spaceStep$ & $\text{{Err-Estim}}_{\heightShallowWater}$ & Order & $\text{{Err-Estim}}_{\velocityShallowWater}$ & Order \\
                    \hline 
                    7.8125E-03 & 	5.8333E-06 &	     &	2.9538E-05 &	    \\
                    3.9063E-03 & 	7.9483E-07 &	2.88 &	1.6474E-06 &	4.16\\
                    1.9531E-03 & 	1.0703E-07 &	2.89 &	4.8759E-08 &	5.08\\
                    9.7656E-04 & 	7.6700E-09 &	3.80 &	2.9001E-09 &	4.07\\
                    4.8828E-04 & 	4.9440E-10 &	3.96 &	1.8273E-10 &	3.99\\
                    2.4414E-04 & 	3.1134E-11 &	3.99 &	1.1456E-11 &	4.00\\
                    1.2207E-04 & 	1.9495E-12 &	4.00 &	7.1665E-13 &	4.00\\
                    6.1035E-05 & 	1.2202E-13 &	4.00 &	4.5492E-14 &	3.98\\                    
                    \hline
                \end{tabular}
            \end{scriptsize}
        \end{center}
\end{table}

We now test the order of the method for a system of equations.
We consider the same setting as \Cref{sec:burgers} except for the fact that we deal with the shallow water system with gravity $\gravityShallowWater = 1$ and initial datum $(\heightShallowWater, \velocityShallowWater) (\timeVariable = 0, \spaceVariable)= (1/2 + 1/5 \sin(2\pi\spaceVariable), 0)$.
Simulations are carried until a final time of $\finalTime = 5/16$ with kinetic velocity $\kineticVelocity = 1.2$.
For the exact solution of the problem is difficult to find, the fourth-order accuracy of the method is demonstrated using the following error estimators:
\begin{multline*}
    \text{{Err-Estim}}_{\heightShallowWater} = \sqrt{\sum_{\indiceSpace\in\relatives}\spaceStep |\heightShallowWater_{\spaceStep}(\finalTime, \indiceSpace\spaceStep) - \heightShallowWater_{\spaceStep/2}(\finalTime, \indiceSpace\spaceStep)|^2}, \\
    \text{{Err-Estim}}_{\velocityShallowWater} = \sqrt{\sum_{\indiceSpace\in\relatives}\spaceStep |\velocityShallowWater_{\spaceStep}(\finalTime, \indiceSpace\spaceStep) - \velocityShallowWater_{\spaceStep/2}(\finalTime, \indiceSpace\spaceStep)|^2},
\end{multline*}
where $\heightShallowWater_{\spaceStep}$ and $\velocityShallowWater_{\spaceStep}$ indicate the discrete solution of our fourth-order scheme computed with space step $\spaceStep$.
We expect fourth-order convergence, which translates into $\text{{Err-Estim}}_{\heightShallowWater}, \text{{Err-Estim}}_{\velocityShallowWater} = \bigO{\spaceStep^4}$ as $\spaceStep \to 0$.
The numerical results in \Cref{tab:convergenceShalloWaters} give the expected trends for both height $\heightShallowWater$ and velocity $\velocityShallowWater$.

\subsection{Multidimensional non-linear scalar problem: Burgers' equation in 2D}

\begin{table}[htbp]
    \caption{\label{tab:convergenceBurgers2D}Error estimations and order of convergence in the $L^2$ metric for the Burgers' equation in 2D.}
    \setlength{\tabcolsep}{2.5pt}
        \begin{center}
            \begin{scriptsize}
                \begin{tabular}{|c|cc|}
                    \hline
                    $\spaceStep$ & $\text{{Err-Estim}}$ & Order  \\
                    \hline
                    6.667E-02	& 8.784E-02	&      \\  
                    3.226E-02	& 2.643E-02	& 1.65 \\
                    1.587E-02	& 9.601E-03	& 1.43 \\
                    7.874E-03	& 1.854E-03	& 2.35 \\
                    3.922E-03	& 3.155E-04	& 2.54 \\
                    1.957E-03	& 3.285E-05	& 3.25 \\
                    9.775E-04	& 1.155E-06	& 4.82 \\
                    4.885E-04	& 2.541E-08	& 5.50 \\
                    2.442E-04	& 1.749E-09	& 3.86\\
                    \hline
                \end{tabular}
            \end{scriptsize}
        \end{center}
\end{table}

To test our approach in 2D, we consider the Burgers' equation $\partial_{\timeVariable}\conservedMomentsSystemConservationLaws + \partial_{x}(\conservedMomentsSystemConservationLaws^2/2) + \partial_{y}(3\conservedMomentsSystemConservationLaws^2/10) = 0$ on the bounded domain $[0, 1]^2$ endowed with periodic boundary conditions.
The initial datum is a narrow Gaussian profile, given by $\conservedMomentsSystemConservationLaws(\timeVariable = 0, \vectorial{\spaceVariable}) = \text{exp}(-100|\vectorial{\spaceVariable} - \transpose{(1/2, 1/2)}|^2)$.
Simulations are carried until final time $\finalTime = 1/16$ where we measure
\begin{equation*}
    \text{{Err-Estim}} = \sqrt{\sum_{\vectorial{\indiceSpace}\in\relatives^{2}}\spaceStep^2 |\conservedMomentsSystemConservationLaws_{\spaceStep}(\finalTime, \vectorial{\indiceSpace} \spaceStep) - \conservedMomentsSystemConservationLaws_{\spaceStep/2}(\finalTime, \vectorial{\indiceSpace} \spaceStep)|^2},
\end{equation*}
given in \Cref{tab:convergenceBurgers2D}.
Once again we observe that our numerical method is fourth-order accurate, as expected.

\section{Entropy stability}\label{sec:entropyStability}

We now address a more useful notion of stability compared to the one studied in \Cref{sec:L2Stability}, which is going to be especially suitable for the non-linear framework.
The total microscopic entropy in the domain---at each time---is given by 
\begin{equation*}
    \sum_{\vectorial{\spaceVariable} \in \spaceLattice{\spatialDimensionality}} \microscopicEntropyLetter(\vectorial{\distributionFunctionLetter}_1(\vectorial{\spaceVariable}), \dots, \vectorial{\distributionFunctionLetter}_{\numberVelocities}(\vectorial{\spaceVariable})).
\end{equation*}
If the computational domain is infinite or periodic boundary conditions on the distribution functions are enforced, this quantity is conserved throughout the transport phase $\transport(\timeStep)$, since it is made up of shifts for each discrete velocity, without mixing the distribution functions between them.
Mathematically, this reads 
\begin{multline*}
    \sum_{\vectorial{\spaceVariable} \in \spaceLattice{\spatialDimensionality}} \microscopicEntropyLetter(\transport(\timeStep)(\vectorial{\distributionFunctionLetter}_1(\vectorial{\spaceVariable}), \dots, \vectorial{\distributionFunctionLetter}_{\numberVelocities}(\vectorial{\spaceVariable}))) \\
    = \sum_{\vectorial{\spaceVariable} \in \spaceLattice{\spatialDimensionality}} \sum_{\indicesVelocity = 1}^{\numberVelocities} \kineticEntropy{\indicesVelocity}(\vectorial{\distributionFunctionLetter}_{\indicesVelocity}(\vectorial{\spaceVariable} - \numberShiftsTransport \tfrac{\vectorial{\kineticVelocity}_{\indicesVelocity}}{\kineticVelocity}\spaceStep)) 
    =  \sum_{\indicesVelocity = 1}^{\numberVelocities} \sum_{\vectorial{\spaceVariable} \in \spaceLattice{\spatialDimensionality}} \kineticEntropy{\indicesVelocity}(\vectorial{\distributionFunctionLetter}_{\indicesVelocity}(\vectorial{\spaceVariable} - \numberShiftsTransport \tfrac{\vectorial{\kineticVelocity}_{\indicesVelocity}}{\kineticVelocity}\spaceStep)) \\
    = \sum_{\indicesVelocity = 1}^{\numberVelocities} \sum_{\vectorial{\spaceVariable} \in \spaceLattice{\spatialDimensionality}} \kineticEntropy{\indicesVelocity}(\vectorial{\distributionFunctionLetter}_{\indicesVelocity}(\vectorial{\spaceVariable})) 
    = \sum_{\vectorial{\spaceVariable} \in \spaceLattice{\spatialDimensionality}} \microscopicEntropyLetter(\vectorial{\distributionFunctionLetter}_1(\vectorial{\spaceVariable}), \dots, \vectorial{\distributionFunctionLetter}_{\numberVelocities}(\vectorial{\spaceVariable})).
\end{multline*}

This is generally not true for the relaxation $\relaxation{\relaxationParameter = 2}$ that we have employed so far.
There is an exception to this: the relaxation phase $\relaxation{\relaxationParameter = 2}$ preserves the microscopic entropy when the problem is linear.
For example, in the context of the linear transport equation, \confer{} \Cref{eq:linearTransportEquation}, simple computations give that 
\begin{equation*}
    \microscopicEntropyLetter (\relaxation{\relaxationParameter = 2}{(\distributionFunctionLetter_+, \distributionFunctionLetter_-)}) = \microscopicEntropyLetter (\distributionFunctionLetter_+, \distributionFunctionLetter_-) = \frac{\kineticVelocity}{\kineticVelocity +  \transportVelocity} (\distributionFunctionLetter_{+})^2 + \frac{\kineticVelocity}{\kineticVelocity -  \transportVelocity} (\distributionFunctionLetter_{-})^2.
\end{equation*}
This is generalized by the following result.

\newcommand{\symetrizer}{\matricial{P}}

\begin{proposition}
    Let \eqref{eq:systemConservationLaws} be linear, that is of the form
    \begin{equation*}
        \partial_{\timeVariable} \vectorial{\conservedMomentsSystemConservationLaws} + \sum_{\indicesSpace = 1}^{\spatialDimensionality}  \matricial{A}^{\indicesSpace}  \frac{\partial \vectorial{\conservedMomentsSystemConservationLaws}}{\partial \spaceVariable_{\indicesSpace}}   = 0.
    \end{equation*}
    Let $\symetrizer$ be a symmetrizer, that is, a symmetric definite positive matrix such that $\symetrizer \matricial{A}^{\indicesSpace}$ is symmetric for all $\indicesSpace \in \integerInterval{1}{\spatialDimensionality}$.
    Consider the natural quadratic entropy-entropy flux given by 
    \begin{equation*}
        \entropy(\vectorial{\conservedMomentsSystemConservationLaws} ) = \frac{1}{2} \symetrizer \vectorial{\conservedMomentsSystemConservationLaws}  \cdot \vectorial{\conservedMomentsSystemConservationLaws} , \qquad \entropyFlux^{\indicesSpace}(\vectorial{\conservedMomentsSystemConservationLaws} ) = \frac{1}{2} \symetrizer \matricial{A}^{\indicesSpace}\vectorial{\conservedMomentsSystemConservationLaws}  \cdot \vectorial{\conservedMomentsSystemConservationLaws}.
    \end{equation*}
    Assume that the kinetic entropies $\kineticEntropy{1}, \dots, \kineticEntropy{\numberVelocities}$ are such that $\kineticEntropyLegendre{1}, \dots, \kineticEntropyLegendre{\numberVelocities}$ are quadratic and convex in their argument, and fulfill
    \begin{equation*}
        \sum_{\indicesVelocity = 1}^{\numberVelocities}\kineticEntropyLegendre{\indicesVelocity}(\vectorial{\dualVariable}) = \legendreTransformed{\entropy}(\vectorial{\dualVariable} ) = \frac{1}{2}\symetrizer^{-1} \vectorial{\dualVariable} \cdot \vectorial{\dualVariable}, \qquad \sum_{\indicesVelocity = 1}^{\numberVelocities}\vectorial{\kineticVelocity}_{\indicesVelocity}^{\indicesSpace}\kineticEntropyLegendre{\indicesVelocity}(\vectorial{\dualVariable}) = \entropyFlux^{\indicesSpace, \almostLegendreTransformedSymbol} (\vectorial{\dualVariable} ) = \frac{1}{2}\matricial{A}^{\indicesSpace}\symetrizer^{-1} \vectorial{\dualVariable} \cdot \vectorial{\dualVariable}.
    \end{equation*}
    Then the relaxation phase $\relaxation{\relaxationParameter = 2}$ conserves the microscopic entropy:
    \begin{equation*}
        \microscopicEntropyLetter (\relaxation{\relaxationParameter = 2}{(\vectorial{\distributionFunctionLetter}_1, \dots, \vectorial{\distributionFunctionLetter}_{\numberVelocities})}) = \microscopicEntropyLetter (\vectorial{\distributionFunctionLetter}_1, \dots, \vectorial{\distributionFunctionLetter}_{\numberVelocities}),
    \end{equation*}
    hence the numerical scheme $\globalScheme$ is entropy preserving.
\end{proposition}
\begin{proof}
    The given quadratic entropy-entropy flux are natural in the sense that 
    \begin{multline*}
        \symetrizer\partial_{\timeVariable} \vectorial{\conservedMomentsSystemConservationLaws} + \sum_{\indicesSpace =  1}^{\spatialDimensionality} \symetrizer \matricial{A}^{\indicesSpace} \partial_{\spaceVariable_{\indicesSpace}} \vectorial{\conservedMomentsSystemConservationLaws}  = 0, \qquad \text{then} \\
         \partial_{\timeVariable} (\symetrizer\vectorial{\conservedMomentsSystemConservationLaws}) \cdot \vectorial{\conservedMomentsSystemConservationLaws} + \sum_{\indicesSpace =  1}^{\spatialDimensionality}  \partial_{\spaceVariable_{\indicesSpace}} (\symetrizer \matricial{A}^{\indicesSpace}\vectorial{\conservedMomentsSystemConservationLaws} )\cdot \vectorial{\conservedMomentsSystemConservationLaws} = \partial_{\timeVariable} \bigl (\tfrac{1}{2}\symetrizer\vectorial{\conservedMomentsSystemConservationLaws}  \cdot \vectorial{\conservedMomentsSystemConservationLaws} \bigr )+ \sum_{\indicesSpace =  1}^{\spatialDimensionality}  \partial_{\spaceVariable_{\indicesSpace}} \bigl (\tfrac{1}{2}\symetrizer \matricial{A}^{\indicesSpace}\vectorial{\conservedMomentsSystemConservationLaws} \cdot \vectorial{\conservedMomentsSystemConservationLaws}\bigr )   = 0.
    \end{multline*}
    The dual entropy and entropy flux can be easily calculated and give the expected constraint on the dual kinetic entropies.
    By assumption, the microscopic entropy $\microscopicEntropyLetter(\vectorial{\distributionFunctionLetter}_{1}, \dots, \vectorial{\distributionFunctionLetter}_{\numberVelocities}) =  \sum_{\indicesVelocity = 1}^{\indicesVelocity = \numberVelocities} \kineticEntropy{\indicesVelocity}(\vectorial{\distributionFunctionLetter}_{\indicesVelocity})$ is a quadratic function in each argument.
    Its minimum under the conservation constraint is given by the equilibria, according to \eqref{eq:entropyDefinitionMinimization}.
    The relaxation $\relaxation{\relaxationParameter = 2}$ given by \eqref{eq:relaxationDiscretization} is nothing but a reflection with respect to the equilibrium and because $\microscopicEntropyLetter$ and its isolines respect this symmetry, the post-relaxation distribution functions yield the same value of $\microscopicEntropyLetter$, concluding the proof.
\end{proof}

\newcommand{\microscopicEntropyImbalance}[2]{\Delta \microscopicEntropyLetter_{(#1)}(#2)}
\newcommand{\microscopicEntropyImbalanceNew}[2]{\Delta \tilde{\microscopicEntropyLetter}_{(#1)}(#2)}

\usetikzlibrary{calc}
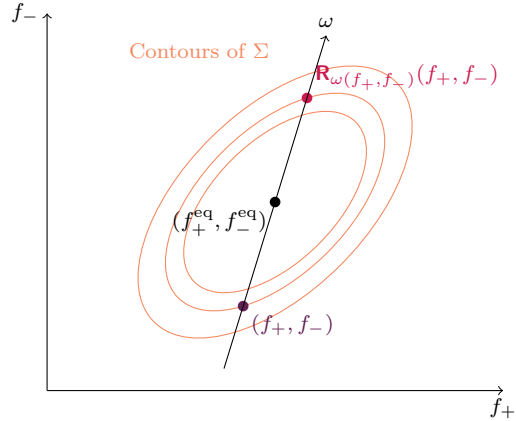
\begin{figure}[htbp]
    \begin{center}
        \begin{footnotesize}
            \begin{tikzpicture}
                \draw[->] (0,0) -- (6,0) node[below] {$\distributionFunctionLetter_+$};
                
                \draw[->] (0,0) -- (0,5) node[left] {$\distributionFunctionLetter_-$};
    
                \draw[color=myOrange, rotate around={45:(3,2.5)}] (3,2.5) ellipse (1.5*1.5 and 1.5*0.8);
                \draw[color=myOrange, rotate around={45:(3,2.5)}] (3,2.5) ellipse (1.2*1.5 and 1.2*0.8);
                \draw[color=myOrange, rotate around={45:(3,2.5)}] (3,2.5) ellipse (1*1.5 and 1*0.8);

                \node[text=myOrange] at (2., 4.5) {Contours of $\microscopicEntropyLetter$};
    
                \coordinate (A) at ($(3,2.5)+(45:1.2*1.5 and 1.2*0.8)$);
                \coordinate(A1)at([rotate around={45:(3,2.5)}]A);
                \fill[myRed](A1) circle (2pt) node[above right]{$\relaxation{\relaxationParameter(\distributionFunctionLetter_+, \distributionFunctionLetter_-)}(\distributionFunctionLetter_+, \distributionFunctionLetter_-)$};
    
                \coordinate (B) at ($(3,2.5)+(-135:1.2*1.5 and 1.2*0.8)$);
                \coordinate(B1)at([rotate around={45:(3,2.5)}]B);
                \fill[color=myViolet](B1) circle (2pt) node[below right]{$(\distributionFunctionLetter_+, \distributionFunctionLetter_-)$};
    
                \fill[black](3,2.5) circle (2pt) node[below left] {$(\distributionFunctionLetter_+^{\atEquilibrium}, \distributionFunctionLetter_-^{\atEquilibrium})$};
    
                \draw[<-] ( $ (A1)!-0.3!(B1) $ ) -- ( $ (A1)!1.3!(B1) $ ) node[pos=0, above] {$\relaxationParameter$};
            \end{tikzpicture}
        \end{footnotesize}
    \end{center}\caption{\label{fig:variableOmegaNew}Idea---illustrated in the case of a two-velocities scheme---behind the procedure selecting a variable relaxation parameter $\relaxationParameter = \relaxationParameter(\distributionFunctionLetter_+, \distributionFunctionLetter_-)$, in order to make the pre and the post-relaxation datum lay on the same level-set of the microscopic entropy. Notice that the equilibrium is a minimizer.}
\end{figure}

However, we are interested in non-linear problems, where is it no longer true that $\relaxation{\relaxationParameter = 2}$ preserves the microscopic entropy.
Our discussion is inspired by \cite{hosseini2023entropic, brownlee2007stability, atif2017essentially}, who however employ a Boltzmann logarithmic entropy of the form $\sum_{\indicesVelocity = 1}^{\indicesVelocity = \numberVelocities}\distributionFunctionLetter_{\indicesVelocity} \log{(\distributionFunctionLetter_{\indicesVelocity} /\omega_{\indicesVelocity})}$, where $\omega_{\indicesVelocity}$ are positive weights, instead of the microscopic entropy $\microscopicEntropyLetter$.
Notice that the Boltzmann logarithmic entropy is well-defined as long as the distribution functions are positive, which is however almost never the case in practice, for they are merely ``numerical'' variables.

The microscopic entropy imbalance through relaxation, very similar to \cite[Equation (4)]{atif2017essentially}, reads 
\begin{equation*} 
    \microscopicEntropyImbalance{\vectorial{\distributionFunctionLetter}_1, \dots, \vectorial{\distributionFunctionLetter}_{\numberVelocities}}{\relaxationParameter} = \microscopicEntropyLetter(\relaxation{\relaxationParameter} {(\vectorial{\distributionFunctionLetter}_1, \dots, \vectorial{\distributionFunctionLetter}_{\numberVelocities})}) - \microscopicEntropyLetter(\vectorial{\distributionFunctionLetter}_1, \dots, \vectorial{\distributionFunctionLetter}_{\numberVelocities}).
\end{equation*}
At each time the relaxation operator is used, for every $\vectorial{\spaceVariable} \in \spaceLattice{\spatialDimensionality}$ and thus for every $\vectorial{\distributionFunctionLetter}_1(\vectorial{\spaceVariable}), \dots, \vectorial{\distributionFunctionLetter}_{\numberVelocities}(\vectorial{\spaceVariable})$, we solve the problem 
\begin{equation}\label{eq:variableOmegaProblem}
    \text{find} \quad \relaxationParameter = \relaxationParameter(\vectorial{\distributionFunctionLetter}_1, \dots, \vectorial{\distributionFunctionLetter}_{\numberVelocities}) \quad \text{such that} \quad \microscopicEntropyImbalance{\vectorial{\distributionFunctionLetter}_1, \dots, \vectorial{\distributionFunctionLetter}_{\numberVelocities}}{\relaxationParameter} = 0,
\end{equation}
and then relax using $\relaxation{\relaxationParameter(\vectorial{\distributionFunctionLetter}_1, \dots, \vectorial{\distributionFunctionLetter}_{\numberVelocities})}$, see \Cref{fig:variableOmegaNew}.
We observe that---even this is practically what happens---there is no guarantee that \eqref{eq:variableOmegaProblem} admits a solution.
This highly depends on the structure of the underlying problem and the chosen entropies.
Notice, and this is very important, that the new relaxation operator is an involution, namely
\begin{equation*}
    \relaxation{\relaxationParameter( \relaxation{\relaxationParameter(\vectorial{\distributionFunctionLetter}_1, \dots, \vectorial{\distributionFunctionLetter}_{\numberVelocities})} (\vectorial{\distributionFunctionLetter}_1, \dots, \vectorial{\distributionFunctionLetter}_{\numberVelocities}))}  \relaxation{\relaxationParameter(\vectorial{\distributionFunctionLetter}_1, \dots, \vectorial{\distributionFunctionLetter}_{\numberVelocities})} (\vectorial{\distributionFunctionLetter}_1, \dots, \vectorial{\distributionFunctionLetter}_{\numberVelocities}) = (\vectorial{\distributionFunctionLetter}_1, \dots, \vectorial{\distributionFunctionLetter}_{\numberVelocities}),
\end{equation*}
which can be seen by looking at \Cref{fig:variableOmegaNew}.
The meaning of this formula is the following: from given distribution functions $(\vectorial{\distributionFunctionLetter}_1, \dots, \vectorial{\distributionFunctionLetter}_{\numberVelocities})$, one computes the non-linear $\relaxationParameter(\vectorial{\distributionFunctionLetter}_1, \dots, \vectorial{\distributionFunctionLetter}_{\numberVelocities})$ by \eqref{eq:variableOmegaProblem}, and relaxes with this rate. Then, the result of this relaxation feeds \eqref{eq:variableOmegaProblem} once again, and a second relaxation with this new rate $\relaxationParameter( \relaxation{\relaxationParameter(\vectorial{\distributionFunctionLetter}_1, \dots, \vectorial{\distributionFunctionLetter}_{\numberVelocities})} (\vectorial{\distributionFunctionLetter}_1, \dots, \vectorial{\distributionFunctionLetter}_{\numberVelocities}))$ is performed. Eventually, the distribution functions go back to their original value $(\vectorial{\distributionFunctionLetter}_1, \dots, \vectorial{\distributionFunctionLetter}_{\numberVelocities})$.
This guarantees that the overall scheme retainsfourth-order accuracy, because the involution property ensures that the basic brick by \eqref{eq:basicBrickSplitting} is time-symmetric, both in the case $\relaxationParameter = 2$ and when $\relaxationParameter = \relaxationParameter(\vectorial{\distributionFunctionLetter}_1, \dots, \vectorial{\distributionFunctionLetter}_{\numberVelocities})$ adapted by \eqref{eq:variableOmegaProblem}.

Finally, we emphasize the fact that enforcing conservation of the microscopic entropy guarantees that the entropy $\entropy$ inside the domain decreases with respect to its initial value during the simulation.
In particular, using \eqref{eq:entropyDefinitionMinimization}, we have:
\begin{align*}
    \sum_{\vectorial{\spaceVariable} \in \spaceLattice{\spatialDimensionality}}\entropy(\vectorial{\conservedMomentsSystemConservationLaws}(\timeVariable, \vectorial{\spaceVariable})) &= \sum_{\vectorial{\spaceVariable} \in \spaceLattice{\spatialDimensionality}} \min_{\vectorial{\conservedMomentsSystemConservationLaws}(\timeVariable, \vectorial{\spaceVariable}) = \sum_{\indicesVelocity = 1}^{\indicesVelocity = \numberVelocities} \vectorial{\distributionFunctionLetter}_{\indicesVelocity}} \microscopicEntropyLetter(\vectorial{\distributionFunctionLetter}_{1}, \dots, \vectorial{\distributionFunctionLetter}_{\numberVelocities}) \\
    &\leq  \sum_{\vectorial{\spaceVariable} \in \spaceLattice{\spatialDimensionality}} \microscopicEntropyLetter(\vectorial{\distributionFunctionLetter}_{1}(\timeVariable, \vectorial{\spaceVariable}), \dots, \vectorial{\distributionFunctionLetter}_{\numberVelocities}(\timeVariable, \vectorial{\spaceVariable})) \\
    &=\sum_{\vectorial{\spaceVariable} \in \spaceLattice{\spatialDimensionality}} \microscopicEntropyLetter(\vectorial{\distributionFunctionLetter}_{1}(0, \vectorial{\spaceVariable}), \dots, \vectorial{\distributionFunctionLetter}_{\numberVelocities}(0, \vectorial{\spaceVariable}))\\
    &=\sum_{\vectorial{\spaceVariable} \in \spaceLattice{\spatialDimensionality}} \microscopicEntropyLetter(\vectorial{\distributionFunctionLetter}_{1}^{\atEquilibrium}(\vectorial{\conservedMomentsSystemConservationLaws}(0, \vectorial{\spaceVariable})), \dots, \vectorial{\distributionFunctionLetter}_{\numberVelocities}^{\atEquilibrium}(\vectorial{\conservedMomentsSystemConservationLaws}(0, \vectorial{\spaceVariable}))) = \sum_{\vectorial{\spaceVariable} \in \spaceLattice{\spatialDimensionality}}\entropy(\vectorial{\conservedMomentsSystemConservationLaws}(0, \vectorial{\spaceVariable})),
\end{align*}
where the last but one equality is valid upon selecting the initial datum at equilibrium, which is the most common choice.

\section{Numerical experiments: Entropy conservation}\label{sec:experimentsEntropy}

The purpose of the numerical experiments in this section is two-fold.
On the one hand, we would like to highlight the importance of the procedure presented in \Cref{sec:entropyStability} to ensure stability.
On the other hand, we want to check that the numerical scheme retains fourth-order accuracy as claimed.

\subsection{Non-linear scalar problem: Burgers' equation in 1D}

We start by considering the Burgers' equation.
We employ the very same setting as \Cref{sec:burgers}, except for the choice $\kineticVelocity = 10$ as far as the kinetic velocity is concerned.
This is done in order to avoid violating the sub-characteristic condition when the first oscillations occur, which would of course drive the simulation to instability and would make the entropy correction unavailable due to the lack of convexity of the kinetic entropies. 
We consider a computational domain made up of 200 points with periodic boundary conditions.
The result is in \Cref{fig:d1q2_n_4_burgers_entropy_correction}, where we use a dichotomy to solve \eqref{eq:variableOmegaProblem} at each point of the lattice.
We see that the simulation which always employs $\relaxation{\relaxationParameter = 2}$ leads to instabilities (we stopped plotting the values when they strongly diverge), whereas the one where $\relaxationParameter$ is adapted using \eqref{eq:variableOmegaProblem} remains stable as claimed.
This is due to the fact that once the shock is formed, the oscillations grow if no entropy correction is used, until some point where the sub-characteristic condition is violated, and instabilities savagely develop.
Furthermore, one sees that when $\relaxationParameter = 2$, the total microscopic entropy steadily increases in time, causing the eventual instability.

\begin{figure}[htbp]
    \begin{center}
        \includegraphics[width = 0.99\textwidth]{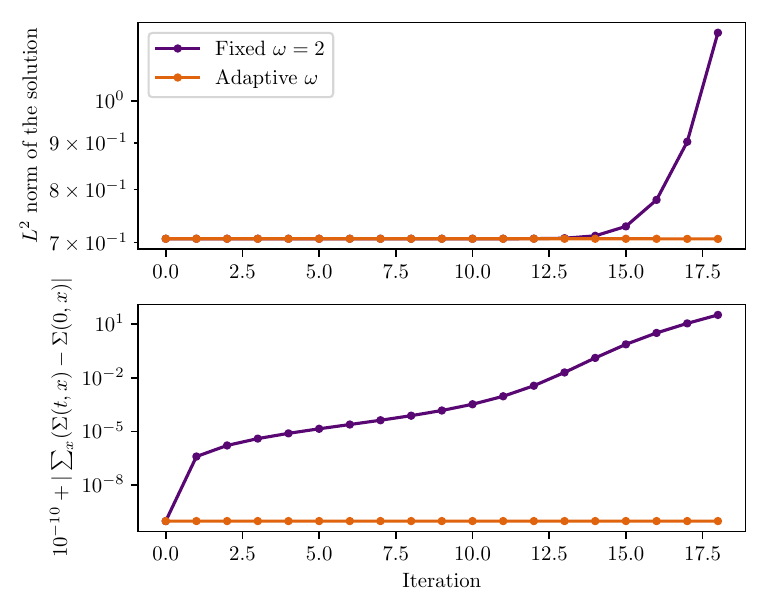}
    \end{center}\caption{\label{fig:d1q2_n_4_burgers_entropy_correction}Norm of the solution (top) and difference between the total microscopic entropy at time zero and eventually in time (bottom), when simulating the Burgers' equation with and without entropy conservation during the relaxation. In the bottom plot, we add $10^{-10}$ to avoid taking the logarithm of zero.}
\end{figure}

Let us point out a practical yet fundamental point on the computation of a solution to \eqref{eq:variableOmegaProblem}.
Indeed, we can have $\vectorial{\distributionFunctionLetter}_{\indicesVelocity} \approx \vectorial{\distributionFunctionLetter}_{\indicesVelocity}^{\atEquilibrium}$, thus making $\microscopicEntropyImbalance{\vectorial{\distributionFunctionLetter}_1, \dots, \vectorial{\distributionFunctionLetter}_{\numberVelocities}}{\relaxationParameter}$ quite close to zero with almost zero derivative in $\relaxationParameter$. 
This frequently cause issues when iterative methods (Newton's, dichotomy, \emph{etc.}) are employed to solve \eqref{eq:variableOmegaProblem}.
The idea is to factorize the distance from the equilibrium in the problem: $\microscopicEntropyImbalance{{\distributionFunctionLetter}_+, {\distributionFunctionLetter}_-}{\relaxationParameter} = (\distributionFunctionLetter_+ - \distributionFunctionLetter_+^{\atEquilibrium})\microscopicEntropyImbalanceNew{{\distributionFunctionLetter}_+, {\distributionFunctionLetter}_-}{\relaxationParameter} = 0$, so that one eventually solves $\microscopicEntropyImbalanceNew{{\distributionFunctionLetter}_+, {\distributionFunctionLetter}_-}{\relaxationParameter} = 0$.

\begin{table}[htbp]
    \caption{\label{tab:convergenceBurgersEntropyCorrection}Errors and order of convergence in the $L^2$ norm for the Burgers' equation using \eqref{eq:variableOmegaProblem}.}
    \setlength{\tabcolsep}{2.5pt}
        \begin{center}
            \begin{scriptsize}
                \begin{tabular}{|c|cc|}
                    \hline
                    $\spaceStep$ & $L^2$ error & Order \\
                    \hline
                    2.000E-03 &		3.725E-06 &	    \\
                    1.250E-03 &		1.764E-06 &	1.59\\
                    7.813E-04 &		1.965E-07 &	4.67\\
                    4.883E-04 &		3.826E-08 &	3.48\\
                    3.053E-04 &		5.898E-09 &	3.98\\
                    1.908E-04 &		9.914E-10 &	3.79\\
                    1.193E-04 &		1.390E-10 &	4.18\\
                    7.454E-05 &		2.410E-11 &	3.73\\
                    \hline
                \end{tabular}
            \end{scriptsize}
        \end{center}
\end{table}

Finally, we check that the entropy conservation procedure \eqref{eq:variableOmegaProblem} does not change the fourth-order convergence of the method.
We operate in the very same setting as in \Cref{sec:burgers}, also re-establishing $\kineticVelocity = 1.2$.
The results in \Cref{tab:convergenceBurgersEntropyCorrection} confirm that no order reduction is experienced and the scheme retains fourth-order accuracy, as claimed.

\subsection{Non-linear system: Shallow water system in 1D}

\begin{figure}[htbp]
    \begin{center}
        \includegraphics[width = 0.99\textwidth]{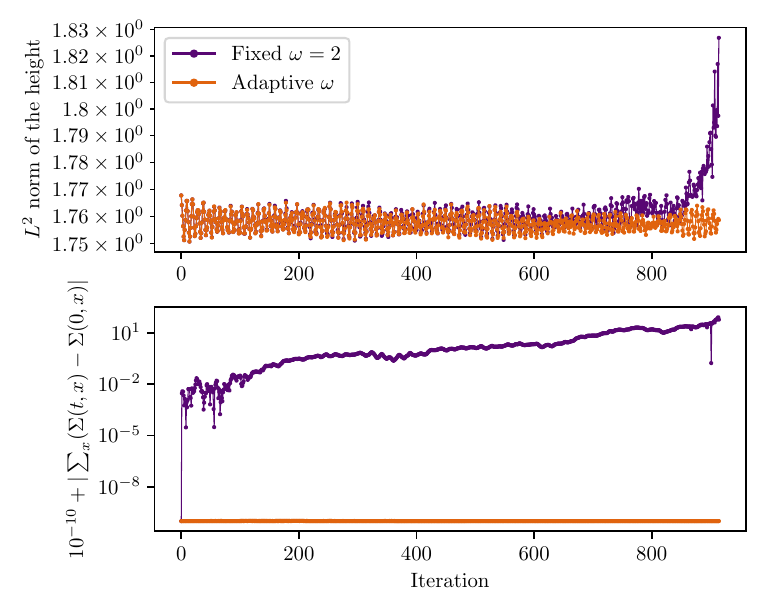}
    \end{center}\caption{\label{fig:shallow_water_explosion}Norm of the solution (height, top) and difference between the total microscopic entropy at time zero and eventually in time (bottom), when simulating the shallow water equations with and without entropy conservation during the relaxation. In the bottom plot, we add $10^{-10}$ to avoid taking the logarithm of zero.}
\end{figure}

For testing the entropy correction on the shallow water system with gravity $\gravityShallowWater = 1$, we take the initial datum 
\begin{equation*}
    (\heightShallowWater, \velocityShallowWater)(\timeVariable = 0, \spaceVariable) = 
    \begin{cases}
        (2, 0), \qquad &\spaceVariable < 1/2, \\
        (3/2, 0), \qquad &\spaceVariable \geq 1/2,
    \end{cases}
\end{equation*}
and the kinetic velocity $\kineticVelocity = 6$ with a spatial grid made up of 100 points.
The results are given in \Cref{fig:shallow_water_explosion}, where problem \eqref{eq:variableOmegaProblem} has been solved using a \emph{quasi}-Newton's method.
This confirms the stabilizing power of the entropy conservation procedure and highlights---once again---that the growth of the total microscopic entropy makes solutions eventually diverge in time.

\section{Variations on the numerical scheme and additional numerical experiments}\label{sec:Variations}

We now propose variations on the basic fourth-order numerical scheme that we have proposed---whose interest is justified---and additional numerical experiments.

\subsection{Projections on the equilibrium}

We consider different schemes---where we introduce projections on the equilibrium at different stages.
This way of proceeding can be used to reduce oscillations and enhance stability when shocks form.
Somehow, these projections can help the numerical scheme to decrease entropy.
Before proceeding, notice that $\relaxation{\relaxationParameter = 1}$ is the projection on the equilibrium.
The name ``projection'' perfectly fits for $\relaxation{\relaxationParameter = 1} \relaxation{\relaxationParameter = 1} = \relaxation{\relaxationParameter = 1}$.
We consider

\begin{align*}
    &\text{Scheme (I)} \qquad 
    &&\begin{cases}
        \globalScheme(\timeStep) \text{ by \eqref{eq:splittingToUse}}, \\
        \basicBrickSolver(\timeStep) \text{ by \eqref{eq:basicBrickSplitting}}. 
    \end{cases}\\
    &\text{Scheme (II)} \qquad 
    &&\begin{cases}
        \globalScheme(\timeStep) = \relaxation{\relaxationParameter = 1}\basicBrickSolver\bigl (\frac{\timeStep}{6}\bigr )^4\basicBrickSolver\bigl (-\frac{\timeStep}{3} \bigr ) \basicBrickSolver\bigl (\frac{\timeStep}{6}\bigr )^4, \\
        \basicBrickSolver(\timeStep) \text{ by \eqref{eq:basicBrickSplitting}}.
    \end{cases}\\
    &\text{Scheme (III)} \qquad 
    &&\begin{cases}
        \globalScheme(\timeStep) = \bigl ( \relaxation{\relaxationParameter = 1} \basicBrickSolver\bigl (\frac{\timeStep}{6}\bigr )\bigl )^4 \relaxation{\relaxationParameter = 1}\basicBrickSolver\bigl (-\frac{\timeStep}{3} \bigr ) \bigl ( \relaxation{\relaxationParameter = 1} \basicBrickSolver\bigl (\frac{\timeStep}{6}\bigr )\bigl )^4, \\
        \basicBrickSolver(\timeStep) \text{ by \eqref{eq:basicBrickSplitting}}.
    \end{cases}\\
    &\text{Scheme (IV)} \qquad 
    &&\begin{cases}
        \globalScheme(\timeStep) \text{ by \eqref{eq:splittingToUse}}, \\
        \basicBrickSolver(\timeStep) = \relaxation{\relaxationParameter = 1}\transport \bigl ( \frac{\splittingTimeStep}{4}\bigr ) \relaxation{\relaxationParameter = 2} \transport \bigl ( \frac{\splittingTimeStep}{4}\bigr ) \relaxation{\relaxationParameter = 1}\transport \bigl ( \frac{\splittingTimeStep}{4}\bigr ) \relaxation{\relaxationParameter = 2}  \transport \bigl ( \frac{\splittingTimeStep}{4}\bigr ).
    \end{cases}
\end{align*}
Let us briefly comment on these schemes.
The first one is the original fourth order scheme we have proposed.
For Scheme (I), (II), and (III), the basic brick $\basicBrickSolver$ is left unchanged.
Scheme (II) just performs a projection on the equilibrium at the end of each fourth-order solver.
Scheme (III) does so after each employ of the basic brick $\basicBrickSolver$.
Finally, Scheme (IV) acts in a radically different fashion, for it combines a modified basic brick $\basicBrickSolver$ with the same composition. The basic brick $\basicBrickSolver$ is modified as a pairing of two Strang formul\ae{} followed by projections on the equilibrium.

\subsection{Non-linear scalar problem: Burgers' equation in 1D}

\begin{table}[htbp]
    \caption{\label{tab:convergenceBurgers}Errors and order of convergence in the $L^2$ norm for the Burgers' equation using different variation on our new fourth-order scheme.}
    \setlength{\tabcolsep}{2.5pt}
        \begin{center}
            \begin{scriptsize}
                \begin{tabular}{|c|cc|cc|cc|cc|}
                    \hline
                    & \multicolumn{2}{c|}{Scheme (I)} & \multicolumn{2}{c|}{Scheme (II)} & \multicolumn{2}{c|}{Scheme (III)} & \multicolumn{2}{c|}{Scheme (IV)} \\
                    \hline
                    $\spaceStep$ & $L^2$ error & Order & $L^2$ error & Order & $L^2$ error & Order  & $L^2$ error & Order \\
                    \hline 
                    \hline
                    \multicolumn{9}{|c|}{Initial datum at equilibrium}\\
                    \hline
                    \hline
                    2.000E-03 & 3.370E-06 & 	     & 3.374E-06 & 	     & 3.246E-06 & 	     & 1.147E-05&	    \\ 
                    1.250E-03 & 1.552E-06 & 	1.65 & 1.551E-06 & 	1.65 & 1.476E-06 & 	1.68 & 5.686E-06&	1.49\\ 
                    7.813E-04 & 1.742E-07 & 	4.65 & 1.742E-07 & 	4.65 & 1.677E-07 & 	4.63 & 1.193E-06&	3.32\\ 
                    4.883E-04 & 3.365E-08 & 	3.50 & 3.365E-08 & 	3.50 & 3.275E-08 & 	3.47 & 3.471E-07&	2.63\\ 
                    3.053E-04 & 5.184E-09 & 	3.98 & 5.184E-09 & 	3.98 & 5.091E-09 & 	3.96 & 8.691E-08&	2.95\\ 
                    1.908E-04 & 8.688E-10 & 	3.80 & 8.688E-10 & 	3.80 & 8.586E-10 & 	3.79 & 2.283E-08&	2.84\\ 
                    1.193E-04 & 1.221E-10 & 	4.18 & 1.221E-10 & 	4.18 & 1.212E-10 & 	4.17 & 5.327E-09&	3.10\\ 
                    7.454E-05 & 2.109E-11 & 	3.74 & 2.109E-11 & 	3.74 & 2.099E-11 & 	3.73 & 1.418E-09&	2.82\\ 
                    \hline 
                    \hline
                    \multicolumn{9}{|c|}{Initial datum off-equilibrium}\\
                    \hline
                    \hline 
                    2.000E-03 &6.432E-06 &	    & 5.355E-06 &	    & 4.115E-05 &       & 1.653E-04 &	     \\ 	
                    1.250E-03 &1.800E-06 &	2.71& 1.639E-06 &	2.52& 1.069E-05 &	2.87& 6.785E-05 &	1.90 \\ 
                    7.813E-04 &2.182E-07 &	4.49& 1.825E-07 &	4.67& 2.576E-06 &	3.03& 2.579E-05 &	2.06 \\ 
                    4.883E-04 &3.945E-08 &	3.64& 3.425E-08 &	3.56& 6.346E-07 &	2.98& 1.015E-05 &	1.98 \\ 
                    3.053E-04 &6.061E-09 &	3.99& 5.234E-09 &	4.00& 1.551E-07 &	3.00& 3.958E-06 &	2.00 \\ 
                    1.908E-04 &9.954E-10 &	3.84& 8.730E-10 &	3.81& 3.796E-08 &	3.00& 1.551E-06 &	1.99 \\ 
                    1.193E-04 &1.423E-10 &	4.14& 1.225E-10 &	4.18& 9.245E-09 &	3.01& 6.033E-07 &	2.01 \\ 
                    7.454E-05 &2.396E-11 &	3.79& 2.112E-11 &	3.74& 2.264E-09 &	2.99& 2.369E-07 &	1.99 \\ 
                    \hline   
                \end{tabular}
            \end{scriptsize}
        \end{center}
\end{table}

To start testing the four different schemes proposed hitherto, we consider exactly the same setting as \Cref{sec:burgers} and we shall test both by initializing the distribution functions at equilibrium and off-equilibrium.
The results are given in \Cref{tab:convergenceBurgers} and show fourth-order convergence---when the initial datum is taken at equilibrium---for all numerical schemes except the last one, where order is reduced to three because of the projection on the equilibrium in the basic brick $\basicBrickSolver$.
Therefore, it is not advisable to employ Scheme (IV).
Trying not to initialize at equilibrium, by taking $(\distributionFunctionLetter_+, \distributionFunctionLetter_-)(\timeVariable = 0) = (\tfrac{1}{4}, \tfrac{3}{4})\conservedMomentsSystemConservationLaws(\timeVariable = 0)$, we observe fourth-order convergence for the first two schemes, third-order for the third scheme, and second-order for the last one, this phenomenon is explained in what follows.
We deduce that Scheme (III) needs to be used carefully, in particular, initializing at equilibrium.

In this non-linear case, the order reduction induced by the projections on the equilibrium could be seen using the modified equations \cite{warming1974modified,guillon:hal-03986533} but this would lead to very tedious calculations.
Alternatively, this phenomenon can be easily understood in the case of linear transport, where $\fluxSystemConservationLaws(\conservedMomentsSystemConservationLaws) = \transportVelocity \conservedMomentsSystemConservationLaws$, using Fourier analysis \cite{strikwerda2004finite}.
In this case, we can look at the expansions of the two roots $\timeShiftOperator_1$ and $\timeShiftOperator_2$ of $\determinant(\timeShiftOperator\identityOperator - \fourierTransformed{\globalScheme}(\timeStep)(\frequency\spaceStep))$ in the limit $|\frequency\spaceStep| \ll 1$, to theoretically understand the different convergence rates.
This provides
\begin{align*}
    \text{Scheme (I)} &\\
    \timeShiftOperator_1(\frequency\spaceStep) &= e^{-i \transportVelocity  \frequency\timeStep} + \frac{i \transportVelocity}{622080} (24\transportVelocity^4-25\transportVelocity^2\kineticVelocity^2 + \kineticVelocity^4) (\frequency\timeStep)^5 + \bigO{|\frequency\timeStep|^6}, \\
    \timeShiftOperator_2(\frequency\spaceStep) &=e^{i \transportVelocity  \frequency\timeStep} - \frac{i \transportVelocity}{622080} (24\transportVelocity^4-25\transportVelocity^2\kineticVelocity^2 + \kineticVelocity^4) (\frequency\timeStep)^5 + \bigO{|\frequency\timeStep|^6}.
\end{align*}
\vspace{-1cm}
\begin{multline*}
    \text{Scheme (II)} \\
    \timeShiftOperator_1(\frequency\spaceStep) = e^{-i \transportVelocity  \frequency\timeStep} + \frac{i \transportVelocity}{622080} (24\transportVelocity^4-25\transportVelocity^2\kineticVelocity^2 + \kineticVelocity^4) (\frequency\timeStep)^5 + \bigO{|\frequency\timeStep|^6},  \\
    \timeShiftOperator_2(\frequency\spaceStep) = 0.
\end{multline*}
\vspace{-1cm}
\begin{multline*}
    \text{Scheme (III)} \\
    \timeShiftOperator_1(\frequency\spaceStep) = e^{-i \transportVelocity  \frequency\timeStep} + \frac{i \transportVelocity}{622080} (24\transportVelocity^4-25\transportVelocity^2\kineticVelocity^2 + \kineticVelocity^4) (\frequency\timeStep)^5 + \bigO{|\frequency\timeStep|^6}, \\
    \timeShiftOperator_2(\frequency\spaceStep) = 0.
\end{multline*}
\vspace{-1cm}
\begin{multline*}
    \text{Scheme (IV)} \\
    \timeShiftOperator_1(\frequency\spaceStep) = e^{-i \transportVelocity  \frequency\timeStep} + \frac{\transportVelocity^2}{3456}(\transportVelocity^2 - \kineticVelocity^2) (\frequency\timeStep)^4 + \bigO{|\frequency\timeStep|^5}, \\
    \timeShiftOperator_2(\frequency\spaceStep) = 0.
\end{multline*}

We see that only the first scheme allows two discrete modes in the system, because no relaxation on the equilibrium is performed.
One mode is the one carrying the accurate part of the solution, whereas $\timeShiftOperator_2(\frequency\spaceStep)$ corresponds to a parasitic numerical mode which is globally fourth-order accurate with respect to a transport equation with opposite velocity $-\transportVelocity$.
Moreover, as all the leading order reminders vanish whenever $\transportVelocity = \kineticVelocity$ and typically increase with $\kineticVelocity$, these expansions suggest that one should take $\kineticVelocity \greatersim \transportVelocity$ but as close as possible to the velocity of the fastest wave, \confer{} \eqref{eq:choiceKineticVelocity}, in order to minimize the truncation errors.
Notice that in the first three schemes, we could observe fifth-order results provided that $|\transportVelocity| = \sqrt{6}/12 \kineticVelocity < \kineticVelocity$.
Indeed, we have even more: since $\timeShiftOperator_1(\frequency\spaceStep) = e^{-i \transportVelocity  \frequency\timeStep} + \tfrac{i \transportVelocity}{622080} (24\transportVelocity^4-25\transportVelocity^2\kineticVelocity^2 + \kineticVelocity^4) (\frequency\timeStep)^5 + \tfrac{\transportVelocity^2}{622080} (24\transportVelocity^4-25\transportVelocity^2\kineticVelocity^2 + \kineticVelocity^4) (\frequency\timeStep)^6 + \bigO{|\frequency\timeStep|^7}$ in the case of Schemes (I) and (II), the method can be sixth-order and this is what we observe through simulations.
For Scheme (III), we have $\timeShiftOperator_1(\frequency\spaceStep) = e^{-i \transportVelocity  \frequency\timeStep} + \tfrac{i \transportVelocity}{622080} (24\transportVelocity^4-25\transportVelocity^2\kineticVelocity^2 + \kineticVelocity^4) (\frequency\timeStep)^5  + \frac{\transportVelocity^2}{1244160} (63\transportVelocity^4 - 65 \transportVelocity^2 \kineticVelocity^2+ 2\kineticVelocity^4) (\frequency\timeStep)^6 + \bigO{|\frequency\timeStep|^7}$, hence the scheme remains only fifth-order accurate when $|\transportVelocity| = \sqrt{6}/12 \kineticVelocity$.

\begin{table}[htbp]
    \caption{\label{tab:convergenceLinear}Errors and order of convergence in the $L^2$ norm for the transport equation with $|\transportVelocity|  = \sqrt{6}/12 \kineticVelocity$ and initial datum at equilibrium.}
    \setlength{\tabcolsep}{2.5pt}
        \begin{center}
            \begin{scriptsize}
                \begin{tabular}{|c|cc|cc|cc|cc|}
                    \hline
                    & \multicolumn{2}{c|}{Scheme (I)} & \multicolumn{2}{c|}{Scheme (II)} & \multicolumn{2}{c|}{Scheme (III)} & \multicolumn{2}{c|}{Scheme (IV)} \\
                    \hline
                    $\spaceStep$ & $L^2$ error & Order & $L^2$ error & Order & $L^2$ error & Order  & $L^2$ error & Order \\
                    \hline
                    5.000E-02 &		2.283E-02& 	     & 2.283E-02 & 	     & 2.307E-02 &       & 1.591E-01 & 	     \\     	
                    3.125E-02 &		1.890E-03& 	5.30 & 1.890E-03 & 	5.30 & 2.660E-03 &	4.60 & 5.120E-02 & 	2.41 \\
                    1.961E-02 &		1.239E-04& 	5.85 & 1.239E-04 & 	5.85 & 2.662E-04 &	4.94 & 1.321E-02 & 	2.91 \\
                    1.235E-02 &		8.066E-06& 	5.90 & 8.066E-06 & 	5.90 & 2.713E-05 &	4.94 & 3.394E-03 & 	2.94 \\
                    7.752E-03 &		5.040E-07& 	5.96 & 5.040E-07 & 	5.96 & 2.685E-06 &	4.97 & 8.506E-04 & 	2.97 \\
                    4.854E-03 &		3.081E-08& 	5.97 & 3.081E-08 & 	5.97 & 2.616E-07 &	4.98 & 2.111E-04 & 	2.98 \\
                    3.040E-03 &		1.855E-09& 	6.00 & 1.855E-09 & 	6.00 & 2.514E-08 &	5.00 & 5.171E-05 & 	3.00 \\
                    1.901E-03 &		1.111E-10& 	6.00 & 1.111E-10 & 	6.00 & 2.406E-09 &	5.00 & 1.265E-05 & 	3.00 \\
                    \hline
                \end{tabular}
            \end{scriptsize}
        \end{center}
\end{table}

These predictions are actually met by the results of \Cref{tab:convergenceLinear}.
They are obtained exactly in the same setting as for the Burgers's equation, taking a final time $\finalTime = 10$ and quite coarse meshes in order to avoid very small errors below machine epsilon in double precision, since the numerical methods are now extremely accurate.
Of course, this is of limited interest since valid only in the linear setting and does not extend to the case of the Burgers' equation.
However, a similar idea could be utilized in the simulation of low-Mach--flows, where the wave speed is roughly constant in the domain, in order to obtain, if not sixth-order schemes, very accurate fourth-order ones.

To understand the order reductions experienced when the initial datum is not at equilibrium, \confer{} the bottom half of \Cref{tab:convergenceBurgers}, again in the linear setting, we follow the procedure by \cite{bellotti2024initialisation}, which has allowed to explain the behavior of the standard second-order lattice Boltzmann scheme as far as initializations are concerned.
Considering that $(\distributionFunctionLetter_+, \distributionFunctionLetter_-)(\timeVariable=0) = (\theta, 1-\theta)\conservedMomentsSystemConservationLaws(\timeVariable = 0)$ with $\theta \in \reals$, we study the low-frequency limit of $\fourierTransformed{\discrete{g}}(\frequency\spaceStep) = (\transpose{\canonicalBasisVector_{1}}+\transpose{\canonicalBasisVector_{2}})\fourierTransformed{\globalScheme}(\timeStep)(\frequency\spaceStep)\transpose{(\theta, 1-\theta)}$, the amplification factor giving the approximation of the conserved variable $\conservedMomentsSystemConservationLaws$ after one time step, as function of the initial datum of the Cauchy problem.
We have 
\begin{equation*}
    \textnormal{Scheme (I) and (II)} \qquad \fourierTransformed{\discrete{g}}(\frequency\spaceStep) = e^{-i \transportVelocity  \frequency\timeStep}  + \bigO{|\frequency\timeStep|^5},
\end{equation*}
independently of $\theta$, which explains why fourth-order is indeed kept.
For the other schemes
\begin{equation*}
    \textnormal{Scheme (III)} \qquad \fourierTransformed{\discrete{g}}(\frequency\spaceStep) = e^{-i \transportVelocity  \frequency\timeStep} - \frac{i \transportVelocity^2}{1728}(\transportVelocity + \kineticVelocity - 2\kineticVelocity\theta)(\frequency\timeStep)^3  + \bigO{|\frequency\timeStep|^4},
\end{equation*}
hence we understand why we observe third-order convergence except when the initial datum is at equilibrium, that is $\theta = \tfrac{1}{2}(1+\tfrac{\transportVelocity}{\kineticVelocity})$.
Finally, we have
\begin{equation*}
    \textnormal{Scheme (IV)} \qquad \fourierTransformed{\discrete{g}}(\frequency\spaceStep) = e^{-i \transportVelocity  \frequency\timeStep} + \frac{\transportVelocity}{288}(\transportVelocity + \kineticVelocity - 2\kineticVelocity\theta)(\frequency\timeStep)^2  + \bigO{|\frequency\timeStep|^3},
\end{equation*}
yielding the same conclusion at second-order.

\subsection{Solution of the Euler equations in 2D with Riemann problems}

\begin{figure}[htbp]
    \begin{center}
        \includegraphics[width = 0.99\textwidth]{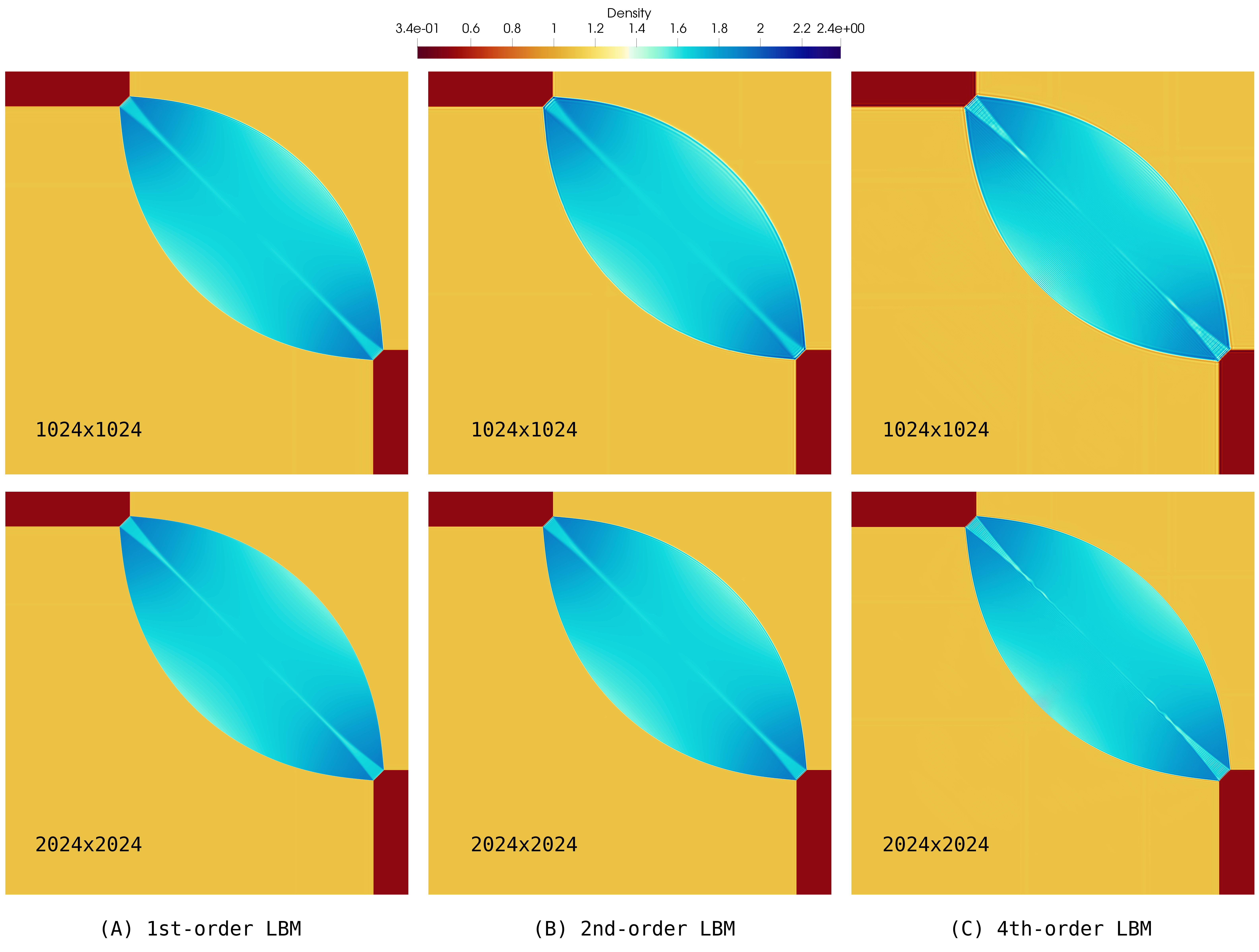}
    \end{center}\caption{\label{fig:euler2D1stround}Densities at final time $\finalTime = 0.25$ with Riemann problems for the Euler equations, using Configuration 4 from \cite{lax1998solution} and employing schemes (A), (B), and (C).}
\end{figure}

We finish the paper on a test concerning the full Euler system in 2D with discontinuous solutions.
Therefore, we have $\spatialDimensionality = 2$, $\numberConservationLaws = 4$ with $\vectorial{\conservedMomentsSystemConservationLaws} = (\conservedMomentsSystemConservationLaws_1, \conservedMomentsSystemConservationLaws_2, \conservedMomentsSystemConservationLaws_3, \conservedMomentsSystemConservationLaws_4) = (\densityEuler, \densityEuler\velocityXEuler, \densityEuler\velocityYEuler, \energyEuler)$ under the fluxes $\vectorial{\fluxSystemConservationLaws}^1(\densityEuler, \velocityXEuler, \velocityYEuler, \pressureEuler) = (\densityEuler\velocityXEuler, \densityEuler\velocityXEuler^2 + \pressureEuler, \densityEuler\velocityXEuler\velocityYEuler, \velocityXEuler(\energyEuler+\pressureEuler))$ and $\vectorial{\fluxSystemConservationLaws}^2(\densityEuler, \velocityXEuler, \velocityYEuler, \pressureEuler) = (\densityEuler\velocityYEuler, \densityEuler\velocityXEuler\velocityYEuler, \densityEuler\velocityYEuler^2 + \pressureEuler, \velocityYEuler(\energyEuler+\pressureEuler))$, where the link between energy $\energyEuler$ and pressure $\pressureEuler$ is given by the polytropic equation of state 
\begin{equation*}
    \energyEuler = \frac{1}{2}\densityEuler (\velocityXEuler^2 + \velocityYEuler^2) + \frac{\pressureEuler}{\gasConstant - 1},
\end{equation*}
with $\gasConstant$ the gas constant. 

We consider an initial datum made up of a Riemann problem with four constant states, given by the configuration 4 from \cite{lax1998solution}, with gas constant $\gasConstant = 1.4$.
The grids are made up of 1024 and 2048 points per direction, with a kinetic velocity $\kineticVelocity = 6.21$.
We would like to compare the second-order standard lattice Boltzmann scheme against our new fourth-order scheme.
In the considered framework, the  standard lattice Boltzmann method $\relaxation{\relaxationParameter}\transport(\timeStep)$ can only be used---for stability reasons---with $\relaxationParameter = 1.93$, hence it is stuck to first-order accuracy, though with little dissipation. We utilize the scheme with $\timeStep = \spaceStep/\kineticVelocity$ and call it scheme (A).
In this case, projecting on the equilibrium at each iteration would result in relaxing with $\relaxationParameter = 1$, which gives extremely diffusive schemes.
We therefore do not make this choice.
To build a second-order stable scheme, which is however not the standard lattice Boltzmann scheme, we consider $\relaxation{\relaxationParameter = 1}\transport \bigl ( \frac{\splittingTimeStep}{4}\bigr ) \relaxation{\relaxationParameter = 2} \transport \bigl ( \frac{\splittingTimeStep}{2}\bigr ) \relaxation{\relaxationParameter = 2}  \transport \bigl ( \frac{\splittingTimeStep}{4}\bigr )$, which boils down to perform one time-step by applying the basic brick \eqref{eq:basicBrickSplitting} plus a relaxation on the equilibrium. We select $\timeStep = 4\spaceStep/\kineticVelocity$ and call the scheme (B).
With our fourth-order scheme, we are able to use $\relaxationParameter = 2$, hence making the scheme fourth-order accurate, provided that we employ Scheme (III).
Indeed, the projection on the equilibrium at each call of the basic brick $\basicBrickSolver$ has a very positive effect on stability.
Notice that one could add shock detection algorithms in order to set $\relaxationParameter = 2$ far from shocks and $\relaxationParameter$ slightly smaller than two, ensuring the elimination of spurious oscillations.
However, as claimed at the very beginning of the manuscript, this is beyond the proof-of-concept status of the paper and shall not be investigated here.
The simulations are not stable when using Scheme (I) and $\relaxationParameter = 2$.

After $\finalTime = 0.25$, the density field is plotted in \Cref{fig:euler2D1stround}.
The results are in agreement with those of \cite{lax1998solution}.
We see that, surprisingly, at the fixed grid sizes that we consider, the lower the order of the scheme, the  sharper the shocks appear on the picture, especially at coarse resolution.
This is due to the (slightly) dissipative character of the scheme of the first-order scheme (A), contrarily to the dispersive nature of the second (B) and fourth-order (C) schemes, which generate visible wiggles around the shocks.
However, this is not very important since, in the case of scalar 1D conservation laws solved with monotone (thus first-order accurate) finite difference schemes, the rate of convergence for solutions with shocks is $\bigO{\sqrt{\spaceStep}}$ in the $L^1$ norm, see \cite{kuznetsov1976accuracy, tang1995sharpness, sabac1997optimal}. 
When the problem is linear, for the same norm, we can expect rates $\bigO{\spaceStep^{2/3}}$ for second-order schemes and $\bigO{\spaceStep^{4/5}}$ for fourth-order ones, see \cite{brenner2006besov}.
Therefore, we cannot hope that either the standard lattice Boltzmann scheme or the 2nd-order scheme beats our fourth-order scheme.
More interestingly, our fourth-order scheme allows observing hydrodynamic instabilities in the smooth area of the flow, along the central axis of the almond-shaped structure. 
This indicates that the underlying numerical scheme is a high-order one, whereas these structures cannot be observed with the standard lattice Boltzmann method and the second-order scheme at the given resolutions.
In terms of computational time on GPUs, the standard lattice Boltzmann algorithm---scheme (A)---took 7.213 seconds to run with 1024 points per direction, and 59.697 seconds with 2048.
Scheme (B) took 6.072 seconds with 1024 and 53.270 seconds with 2048 points.
The fourth-order method---scheme (C)---took 9.470 seconds with 1024  and 64.529 seconds with 2048 points.

\section{Conclusions}\label{sec:conclusions}

In this study, we have introduced a general framework for constructing fourth-order lattice Boltzmann schemes tailored to handle hyperbolic systems of conservation laws as long as their solution remains smooth.
Our procedure relies on time-symmetric operators, combined together to increase the order of the method.
For we employ a kinetic relaxation approximation, we can adjust the kinetic velocities to ensure that the resulting scheme adheres to the lattice Boltzmann method principles. Numerical simulations have been conducted to validate our theoretical findings. Furthermore, we have proposed modifications to the local relaxation phase to maintain entropy stability without compromising the  order of the method.

Future research directions include the development of limitation strategies---both \emph{a priori} and \emph{a posteriori}---for these lattice Boltzmann schemes to effectively address numerical oscillations arising from shocks. Additionally, techniques to ensure positivity, particularly when dealing with shallow water equations, will be of interest.
Finally, exploring methods to further increase the order of the scheme, potentially up to six or beyond, holds promise for enhancing accuracy and computational efficiency.

\section*{Acknowledgments}

This work of the Interdisciplinary Thematic Institute IRMIA++, as part of the ITI 2021-2028 program of the University of Strasbourg, CNRS and Inserm, was supported by IdEx Unistra (ANR-10-IDEX-0002), and by SFRI-STRAT’US project (ANR-20-SFRI-0012) under the framework of the French Investments for the Future Program.

The authors thank the two anonymous referees for their useful suggestions to improve the work.    
T. Bellotti also thanks O. A. Boolakee from ETH Zürich for having pointed out a typo in the first version of the manuscript.

\bibliographystyle{siamplain}
\bibliography{biblio}

\appendix

\section{Trace of the amplification matrix for the linear \lbmSchemeVectorial{1}{2}{} scheme}\label{sec:traceD1Q2}

Introducing $C = \transportVelocity\timeStep/(\numberShiftsTransport\spaceStep)$ and $\mu(\frequency\spaceStep) = \sin(\numberShiftsTransport\frequency\spaceStep) \in [-1, 1]$, we have
    \begin{align*}
        \trace (\fourierTransformed{\globalScheme}(\timeStep)(\frequency\spaceStep)) = 
        -\tfrac{{\left(C^{18} - 576  C^{16}\right)}1}{101559956668416}   \mu(\frequency\spaceStep)^{36} + \tfrac{17{\left(C^{18} - 576  C^{16}\right)}}{203119913336832}   \mu(\frequency\spaceStep)^{34} &\\
        - \tfrac{{\left(8  C^{18} - 4491  C^{16} - 67392  C^{14}\right)}1}{25389989167104}   \mu(\frequency\spaceStep)^{32} + \tfrac{{\left(35  C^{18} - 18414  C^{16} - 1005696  C^{14}\right)}}{50779978334208}   \mu(\frequency\spaceStep)^{30} &\\
        - \tfrac{{\left(49  C^{18} - 22554  C^{16} - 3223152  C^{14} - 24634368  C^{12}\right)} }{50779978334208}  \mu(\frequency\spaceStep)^{28} &\\
        + \tfrac{{\left(91  C^{18} - 31500  C^{16} - 11499408  C^{14} - 315767808  C^{12}\right)} }{101559956668416}  \mu(\frequency\spaceStep)^{26} &\\
        - \tfrac{{\left(14  C^{18} - 2079  C^{16} - 3074112  C^{14} - 213077952  C^{12} - 1108546560  C^{10}\right)} }{25389989167104}  \mu(\frequency\spaceStep)^{24} &\\
        + \tfrac{{\left(11  C^{18} + 2358  C^{16} - 3889944  C^{14} - 623464128  C^{12} - 11824496640  C^{10}\right)}}{50779978334208}   \mu(\frequency\spaceStep)^{22} &\\
        - \tfrac{{\left(5  C^{18} + 4932  C^{16} - 2516832  C^{14} - 1053025920  C^{12} - 51399608832  C^{10} - 203166351360  C^{8}\right)} }{101559956668416}  \mu(\frequency\spaceStep)^{20} &\\
        + \tfrac{{\left(C^{18} + 3384  C^{16} - 204768  C^{14} - 992466432  C^{12} - 115473600000  C^{10} - 1712402104320  C^{8}\right)}}{203119913336832}   \mu(\frequency\spaceStep)^{18} &\\
        - \tfrac{{\left(C^{16} + 792  C^{14} - 494208  C^{12} - 160807680  C^{10} - 6429570048  C^{8} - 20316635136  C^{6}\right)}}{470184984576}   \mu(\frequency\spaceStep)^{16} &\\
        + \tfrac{{\left(3  C^{14} + 8  C^{12} - 871200  C^{10} - 91228032  C^{8} - 1128701952  C^{6}\right)}}{8707129344}   \mu(\frequency\spaceStep)^{14} &\\
        - \tfrac{{\left(3  C^{12} - 1100  C^{10} - 445392  C^{8} - 15894144  C^{6} - 41803776  C^{4}\right)} }{120932352}  \mu(\frequency\spaceStep)^{12}&\\
         + \tfrac{ {\left(11  C^{10} - 9720  C^{8} - 997920  C^{6} - 10450944  C^{4}\right)} }{20155392} \mu(\frequency\spaceStep)^{10} 
         + \tfrac{ {\left(C^{8} + 224  C^{6} + 2496  C^{4} + 13824  C^{2}\right)} }{31104} \mu(\frequency\spaceStep)^{8} &\\
         - \tfrac{{\left(7  C^{6} - 24  C^{4} - 2304  C^{2}\right)} }{2592}  \mu(\frequency\spaceStep)^{6} + \tfrac{{\left(C^{4} - 4  C^{2}\right)} }{12}  \mu(\frequency\spaceStep)^{4} - C^{2} \mu(\frequency\spaceStep)^{2} + 2.
    \end{align*}

\begin{figure}[htbp]
    \begin{center}
        \includegraphics[width = 0.99\textwidth]{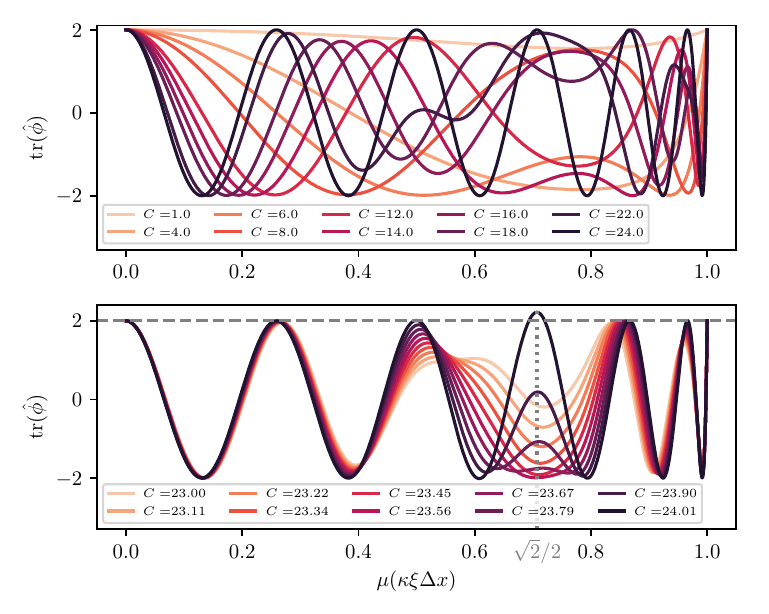}
    \end{center}\caption{\label{fig:traced1q2}Trace of the amplification matrix for the linear \lbmSchemeVectorial{1}{2}{} scheme for different values of $C$.}
\end{figure}

This is an even polynomial of degree 36 to study on a bounded set $[0, 1]$.
Since only even powers of $C$ appear, we can just analyze $C>0$.
From \Cref{fig:traced1q2}, we are confident that the polynomials stay in the band $[-2, 2]$ as long as $C\leq 24$, and they depart from it through a maximum forming at $\mu(\frequency\spaceStep) = \sqrt{2}/2$, namely for $\frequency\spaceStep = \pi/(4\numberShiftsTransport)$.

\end{document}